\documentclass[letter,final]{amsart}
\usepackage{url}
\usepackage[english]{babel}
\usepackage{amsmath,amsfonts,amssymb,amsthm}
\usepackage{latexsym,stmaryrd}
\usepackage{enumerate}
\usepackage{framed}
\usepackage{xypic}
\usepackage{pstricks}

\newcommand{\N}{\mathbb{N}}
\newcommand{\Z}{\mathbb{Z}}

\newcommand{\R}{\mathbb{R}}

\newcommand{\F}{\mathbb{F}}
\newcommand{\calO}{\mathcal{O}}
\newcommand{\frakP}{\mathfrak{P}}

\newcommand{\fraka}{\mathfrak{a}}

\newcommand{\frakm}{\mathfrak{m}}
\newcommand{\frako}{\mathfrak{o}}
\newcommand{\frakp}{\mathfrak{p}}

\newcommand{\abs}[1]{{\left|{#1}\right|}}
\newcommand{\ggen}[1]{{\left\langle{#1}\right\rangle}}
\newcommand{\floor}[1]{{\left\lfloor{#1}\right\rfloor}}
\newcommand{\ceil}[1]{{\left\lceil{#1}\right\rceil}}
\newcommand{\Matrix}[1]{{\left(\begin{matrix} #1 \end{matrix}\right)}}

\DeclareMathOperator{\order}{ord}
\DeclareMathOperator{\Pic}{Pic}
\DeclareMathOperator{\lcm}{lcm}
\DeclareMathOperator{\Norm}{Norm}

\newtheorem{theorem}{Theorem}[section]
\newtheorem{corollary}{Corollary}
\newtheorem{lemma}[theorem]{Lemma}
\newtheorem{proposition}{Proposition}

\theoremstyle{definition}

\newenvironment{enuma}{\begin{enumerate}[\upshape (a)]}{\end{enumerate}}
\newenvironment{enumi}{\begin{enumerate}[\upshape (i)]}{\end{enumerate}}
\newenvironment{enum1}{\begin{enumerate}[\upshape (1)]}{\end{enumerate}}

\newcommand{\alginput}[1]{\item \textbf{Input:} #1}
\newcommand{\algoutput}[1]{\item \textbf{Output:} #1}
\newenvironment{algorithm}[1]
{\refstepcounter{theorem}
\setlength{\FrameSep}{5pt}\begin{MakeFramed}{}\begin{list}{}{\setlength{\topsep}{0cm}%
    \setlength{\parsep}{\parskip}\setlength{\leftmargin}{0cm}\setlength{\rightmargin}{1.75cm}}%
    \item \textbf{Algorithm \thetheorem: #1} \hrule}
    {\end{list}\end{MakeFramed}}

\title{Explicit Methods for Radical Function Fields over Finite Fields}
\author{Felix Fontein}
\address{Department of Mathematics \&\ Statistics \\ University of Calgary \\ 2500 University Drive
NW \\ Calgary, Alberta, Canada T2N 1N4}
\curraddr{}
\email{felix.fontein@math.uzh.ch}
\thanks{This work has been supported in part by the Swiss National Science Foundation under grant
no.~107887.}



\DeclareMathOperator{\sgn}{sgn}
\DeclareMathOperator{\disc}{disc}

\DeclareMathOperator{\Gl}{Gl}
\DeclareMathOperator{\Diff}{Diff}
\DeclareMathOperator{\mymod}{mod}
\DeclareMathOperator{\adj}{adj}
\newcommand{\PP}{\mathbb{P}}

\begin{document}
  \maketitle
  
  \begin{abstract}
    We develop explicit formulas and algorithms for arithmetic in radical function fields~$K/k(x)$
    over finite constant fields. First, we classify which places of $k(x)$ whose local integral
    bases have an easy monogenic form, and give explicit formulas for these bases. Then, for a fixed
    place $\frakp$ of $k(x)$, we give formulas for functions whose valuation is zero for all places
    $\frakP \mid \frakp$ except one, for which it is one. We extend a result by Q.~Wu on a
    $k[x]$-basis of its integral closure in $K$, show how to compute certain Riemann-Roch spaces and
    how to compute the exact constant field, resulting in explicit formulas for the exact constant
    field together with easy to evaluate formulas for the genus of $K$. Finally, we show how to
    approximate the Euler product to obtain the class number using ideas of R.~Scheidler and
    A.~Stein and give an algorithm. We give bounds on the running time for all algorithms.
  \end{abstract}
  
  \section{Introduction}
  
  There exist a lot of very general algorithms to perform explicit arithmetic in global function
  fields; for a good overview, see \cite{diem-habil}. In theory, all arithmetic operations have a
  running time polynomially bounded in $\log q$ and certain other invariants, but in practice, these
  algorithms are often slow compared to more specialized solutions.
  
  For example, if one compares how general methods for arithmetic perform in elliptic function
  fields, it turns out that these methods are extremely slower than working with points on the
  corresponding elliptic curve instead. Besides elliptic function fields, one also has very efficient
  and optimized arithmetic for hyperelliptic function fields \cite{hehcc}. Besides these, there are
  other classes of function fields for which specialized arithmetic exists, for example, cubic
  function fields \cite{mark-arithmeticcubicfields, scheidler-infrastructurepurelycubic}, $C_{ab}$
  function fields and superelliptic function fields \cite{galbraith-paulus-smart}.
  
  In this paper, we will concentrate on radical function fields, i.e. function fields of the form~$K
  = k(x, y)$, where $y$ satisfies an equation of the form $y^n - D$ with $D \in k(x)$ and $n$ is not
  divisible by the characteristic of $k$. Hyperelliptic curves in characteristic~$\neq 2$ are a
  special case of radical function fields, as well as superelliptic function fields: the latter are
  radical function fields with one place at infinity and where $D$ is a squarefree
  polynomial. Hence, our methods extend results for these special cases.
  
  In the case of radical function fields over finite constant fields, not much work has been done in
  the direction of explicit arithmetic. One notable exception is a result by Q.~Wu, which gives an
  explicit $k[x]$-basis of the integral closure~$\calO$ of $k[x]$ in $K$
  \cite{wu-radicalffsintbases} under the assumption that $D \in k[x]$ is $n$-th power free. We will
  reformulate his result in Section~\ref{intbasespart2} to work for all $D \in k(x)$.
  
  To do explicit arithmetic in $K$, we present algorithms which compute local as well as global
  integral bases. The goal is that these bases are as explicit and simple as possible. For most
  places~$\frakp$ of $k(x)$, one can give a very simple monogenic basis of the integral
  closure~$\calO_\frakp'$ of $\frako_\frakp$ in $K$, i.e. one give an element of the form $\rho =
  y^i/\pi^j$ such that $\calO_\frakp' = \frako_\frakp[\rho]$; here, $\pi$ is a uniformizer for
  $\frakp$. We use this to give an easy algorithm for computing all places~$\frakP$ of $K$ lying
  above a place $\frakp$ of $k(x)$. Moreover, we find elements $f_\frakP \in K^*$ for $\frakP \mid
  \frakp$ of a simple form such that $\nu_\frakP(f_\frakP) = 1$ and $\nu_{\frakP'}(f_\frakP) = 0$
  for $\frakP' \neq \frakP$, $\frakP' \mid \frakp$.
  
  These methods allow to use the algorithm of F.~He\ss\ to compute Riemann-Roch spaces
  \cite{hessRR}, and our methods allow to give an explicit bound on the running time of the
  algorithm assuming that the divisor is given in form of a $k[x]$-basis of a fractional
  $\calO$-ideal together with integers for the infinite valuations.
  
  We then use the theory for Riemann-Roch space computations to compute the exact constant field of
  $K/k(x)$ as the Riemann-Roch space of the zero divisor. This results in an explicit criterion when
  $k$ is the exact field of constants, and furthermore we give an explicit $k$-basis of the exact
  constant field~$k'$ and an explicit and easy to evaluate formula for the degree~$[k' : k]$. This,
  in turn, allows us to give an explicit formula for the genus of $K/k(x)$.
  
  Finally, we apply the Euler product approximation of R.~Scheidler and A.~Stein
  \cite{scheidlerstein-approxeulerprods} to radical function fields. We reformulate their
  approximation of the class number to make it better suited for numerical evaluation, and provide
  explicit algorithms to compute the approximation. Our discussion includes a bound on the running
  time in binary operations.
  
  All algorithms in this paper, except the Euler product approximation in
  Section~\ref{eulerprodapprox}, have been implemented by the author in C++, and are used to do
  explicit arithmetic in the divisor class group of a radical function field using infrastructure
  methods (see \cite{ff-tioagfoaur}).
  
  \subsection{Notation}
  \label{notation}
  Let $k$ be a field and $n \in \N$, $n > 1$ coprime to the characteristic of $k$. Let $D \in
  k(x)^*$ such that $Y^n - D \in k(x)[Y]$ is irreducible; then $D \neq \alpha^t$ for all $\alpha \in
  k(x)$ and all divisors~$t$ of $n$, $t > 1$. Let $K = k(x, y)$, where $y$ is a root of $Y^n - D$.
  
  Write $D = \sgn(D) \cdot \prod_{i=1}^\infty \frac{f_i^i}{f_{-i}^i}$ with $\sgn(D) \in k^*$ and
  $\dots, f_{-2}, f_{-1}, f_1, f_2, \dots$ a sequence of pairwise coprime squarefree monic
  polynomials, almost all of them being 1. For convenience, define $f_0 := 1$. Note that the
  condition $D \neq \alpha^t$ for all $\alpha \in k(x)$ is equivalent to $\sgn(D)$ not being a
  $t$-th power or $f_i \neq 1$ for some $i \in \Z$ with $t \nmid i$. Moreover, note that checking
  whether an element is a $t$-th power in $\F_q$ can be effectively done; see
  Corollary~\ref{rootcheckcorollary}.
  
  We denote the set of places of a function field~$K'$ by $\PP_{K'}$. For $\frakp \in \PP_{K'}$, let
  $\nu_\frakp : K' \to \Z \cup \{ \infty \}$ be the surjective valuation of $\frakp$. If $K''/K'$ is
  an extension of function fields and $\frakp \in \PP_{K'}$, $\frakP \in \PP_{K''}$, we write
  $\frakP \mid \frakp$ if, and only if, $\frakP \cap K' = \frakp$.
  
  For a place $\frakp \in \PP_{k(x)}$, let $\frako_\frakp$ be the valuation ring of $\frakp$ with
  maximal ideal $\frakm_\frakp$ and let $\calO_\frakp'$ be the integral closure of $\frako_\frakp$
  in $K$. Moreover, write $\kappa(\frakp) := \calO_\frakp / \frakm_\frakp$ for the residue field of
  $\frakp$.
  
  For a place $\frakP \in \PP_K$, let $\calO_\frakP$ be the valuation ring with maximal ideal
  $\frakm_\frakP$. If $\frakp = \frakP \cap k(x)$, let $e(\frakP \mid \frakp) =
  \frac{\nu_\frakP(t)}{\nu_\frakp(t)}$ be the ramification index of $\frakP$ over $\frakp$ (where $t
  \in k(x)^*$ satisfies $\nu_\frakp(t) \neq 0$) and $f(\frakP \mid \frakp)$ the extension degree
  $[\calO_\frakP/\frakm_\frakP : \kappa(\frakp)]$.
  
  Let $\frakp_\infty$ be the infinite place of $k(x)$, i.e. the one whose valuation is given by
  $\nu_{\frakp_\infty}(\frac{f}{g}) = \deg g - \deg f$ for $f, g \in k[x]$, $g \neq 0$; this place
  is the only pole of $x$. We use the abbrevations $\frako_\infty := \frako_{\frakp_\infty}$ and
  $\calO_\infty := \calO_{\frakp_\infty}'$. Moreover, we denote by $\calO$ the integral closure of
  $k[x]$ in $K$. Finally, we call the places of $K$ above $\frakp_\infty$ the \emph{infinite places}
  of $K$; all other places are called \emph{finite places}.
  
  \section{Monogenic Integral Bases and Splitting of Primes}
  \label{monogenicintbasessect}
  
  In this section, we develop a criterion when a place~$\frakp \in \PP_{k(x)}$ with uniformizer~$\pi
  \in k(x)^*$ possesses a monogenic integral basis of $\calO_\frakp'$ of the form $y^i \pi^j$ with
  $i, j \in \Z$ in Proposition~\ref{monogenicintbases}. Moreover, we show how a local integral basis
  can be constructed in any case. Then, we show how to find elements in $K^*$ which have valuation 1
  for one place $\frakP \mid \frakp$ and valuation 0 for all other places lying above $\frakp$ in
  Proposition~\ref{placedecomposition}.
  
  We begin with a result on monogenic integral bases and the places of $K$ lying above a place of
  $k(x)$.
  
  \begin{proposition}
    \label{monogenicintbases}
    Let $\frakp \in \PP_{k(x)}$ and $d := \gcd(n, \nu_\frakp(D))$. Then there exists an element~$z
    \in K$ such that $\calO_\frakp' = \frako_\frakp[z]$ and $z^n \in k(x)$ if, and only if, $d \in
    \{ 1, n \}$.
    
    To be more precisely, let $\pi \in \frako_\frakp$ be a uniformizer for $\frakp$. Then we have:
    \begin{enuma}
      \item If $d = 1$, let $a, b \in \Z$ with $a \nu_\frakp(D) + b n = 1$. Then $z := y^a \pi^b$
      satisfies $\calO_\frakp' = \frako_\frakp[z]$. The minimal polynomial of $z$ over $k(x)$ is
      $Z^n - D^a \pi^{n b}$.
      
      Moreover, $\frakp$ totally ramifies in $K$, i.e. there is exactly one place~$\frakP \in \PP_K$
      lying above $\frakp$ and $e(\frakP | \frakp) = n$, $f(\frakP | \frakp) = 1$. Finally,
      $\nu_\frakP(z) = 1$.
      \item If $d = n$, let $b \in \Z$ with $n b = -\nu_\frakp(D)$. Then $z := y \pi^b$ satisfies
      $\calO_\frakp' = \frako_\frakp[z]$. The minimal polynomial of $z$ over $k(x)$ is $Z^n - D
      \pi^{-\nu_\frakp(D)}$.
      
      Moreover, $\frakp$ is unramified in $K$, i.e. all places $\frakP \in \PP_K$ lying over
      $\frakp$ satisfy $e(\frakP | \frakp) = 1$. The degrees of the places are determined by the
      factorization of $Z^n - \alpha$ over $\kappa(\frakp) = \frako_\frakp / \frakm_\frakp$, where
      $\alpha = D \pi^{b n} + \frakm_\frakp \neq 0$. Finally, $\nu_\frakP(z) = 0$ for all $\frakP$
      lying above $\frakp$.
      \item If $1 < d < n$, the ramification indices of the places $\frakP \in \PP_K$ lying above
      $\frakp$ are given by $\frac{n}{d}$.
    \end{enuma}
    In any case, an integral basis of $\calO_\frakp'$ is given by \[ \left( \pi^{-\floor{i
    \frac{\nu_\frakp(D)}{n}}} y^i \right)_{i=0,\dots,n-1}. \]
  \end{proposition}
  
  The results on the ramification and splitting are well known; see, for example, \cite[p.~111f,
  Proposition~III.7.3]{stichtenoth}.
  
  \begin{proof}
    By \cite[p.~111, Proposition~III.7.3~(b)]{stichtenoth}, $e(\frakP | \frakp) = \frac{n}{d}$ and
    $d(\frakP | \frakp) = \frac{n}{d} - 1$ for all places $\frakP$ lying above $\frakp$. This shows
    that $\nu_\frakP(y) = \frac{1}{n} \nu_\frakP(y^n) = \frac{1}{n} \nu_\frakP(D) = \frac{1}{n}
    e(\frakP | \frakp) \nu_\frakp(D) = \frac{1}{d} \nu_\frakp(D)$ for any place $\frakP$ lying above
    $\frakp$.
    
    Now, let us consider the three cases.
    \begin{enuma}
      \item Let $d = 1$; then $e(\frakP \mid \frakp) = n$, whence $\frakp$ totally ramifies in
      $K$. Let $a, b \in \Z$ with $a \nu_\frakp(D) + b n = 1$ and set $z := y^a \pi^b$. Clearly,
      $k(x)(z) = k(x)(y)$ as $a$ is coprime to $n$, whence the minimal polynomial has
      degree~$n$. Now $z^n = (y^n)^a \pi^{b n} = D^a \pi^{b n} \in k(x)$, whence $\varphi := Z^n -
      D^a \pi^{b n}$ is the minimal polynomial.
      
      Thus, we have $\nu_\frakP(z) = a \nu_\frakP(y) + b \nu_\frakP(\pi) = a \nu_\frakp(D) + b
      e(\frakP | \frakp) = 1$. Now $\varphi'(z) = n z^{n - 1}$, whence $\nu_\frakP(\varphi'(z)) = n
      - 1 = d(\frakP | \frakp)$. Therefore, by \cite[p.~96, Theorem~III.5.10]{stichtenoth},
      $\calO_\frakp' = \frako_\frakp[z]$.
      \item Let $d = n$. If $b = -\frac{\nu_\frakp(D)}{n}$ and $z := y \pi^b$, then $\nu_\frakP(z) =
      \nu_\frakP(y) - \frac{\nu_\frakp(D)}{n} \nu_\frakP(\pi) = \frac{1}{d} \nu_\frakp(D) -
      \frac{1}{d} \nu_\frakp(D) = 0$. Moreover, clearly $k(x)(y) = k(x)(z)$ as $\frac{z}{y} \in
      k(x)$. Hence, the minimal polynomial is given by $\varphi := Z^n - z^n = Z^n - D \pi^{b n} =
      Z^n - D \pi^{-\nu_\frakp(D)}$. Again, $\varphi'(z) = n z^{n - 1}$, whence
      $\nu_\frakP(\varphi'(z)) = 0 = e(\frakP | \frakp) - 1$. Therefore, by \cite[p.~96,
      Theorem~III.5.10]{stichtenoth}, $\calO_\frakp' = \frako_\frakp[z]$.
      
      By Kummer's Theorem \cite[p.~76, Theorem~III.3.7]{stichtenoth}, the factorization of $Z^n - (D
      \pi^{-\nu_\frakp(D)} + \frakm_\frakp)$ over $\kappa(\frakp)[Z]$ gives the places of $K$ lying
      above $\frakp$.
      \item Let $1 < d < n$. Assume that a $z$ exists with minimal polynomial $Z^n - \alpha$ such
      that $\calO_\frakp' = \frako_\frakp[z]$. Let $P : \frako_\frakp \to \kappa(\frakp)$ the
      projection. Then, by Kummer's Theorem \cite[p.~76, Theorem~III.3.7]{stichtenoth}, the
      factorization of $Z^n - P(\alpha) \in \kappa(\frakp)[Z]$ determines the ramification indices
      and relative degrees of the places of $K$ lying above $\frakp$.
      
      If $P(\alpha) = 0$, the polynomial factors as $Z^n$, whence $\frakp$ totally ramifies in $K$:
      but then $n = e(\frakP | \frakp) = \frac{n}{d}$, whence $d = 1$, a contradiction.
      
      In case $P(\alpha) \neq 0$, the polynomial $Z^n - P(\alpha)$ is squarefree as $n$ is coprime
      to the characteristic of $\kappa(\frakp)$. Thus, $1 = e(\frakP | \frakp) = \frac{n}{d}$ for
      all place $\frakP$ lying above $\frakp$: therefore, $d = n$, a contradiction.
      
      Thus, if $1 < d < n$, such a $z$ cannot exist.
    \end{enuma}
    
    Finally, we want to show that $\calO_\frakp' = \bigoplus_{i=0}^{n-1} \frako_\frakp z_i$, where
    $z_i := \pi^{-\floor{i \frac{\nu_\frakp(D)}{n}}} y^i$. For that, we use a similar argument chain
    as in \cite[Section~3]{wu-radicalffsintbases}, which simplifies a lot in this special case. Let
    $\frakP$ be a place lying above $\frakp$. First,
    \begin{align*}
      \nu_\frakP\left( z_i \right) ={} & \frac{1}{n}
      \nu_\frakP\left( \pi^{-\floor{i \frac{\nu_\frakp(D)}{n}} n} y^{i n} \right) \\
      {}={} & \frac{1}{d} \left( -\floor{\frac{i \nu_\frakp(D)}{n}} n + i \nu_\frakp(D) \right).
    \end{align*}
    Now $\floor{\frac{i \nu_\frakp(D)}{n}} n = i \nu_\frakp(D) - ((i \nu_\frakp(D)) \mod n)$, whence
    $\nu_\frakP(z_i) = \frac{((i \nu_\frakp(D)) \mod n)}{d} \ge 0$. Therefore, $z_i \in
    \calO_\frakp'$. Now \[ \disc(z_0, \dots, z_n) = \disc(1, y, \dots, y^{n-1}) \cdot
    \prod_{i=0}^{n-1} \pi^{-2 \floor{i \frac{\nu_\frakp(D)}{n}}}, \] whence \[ \nu_\frakp(\disc(z_0,
    \dots, z_n)) = \nu_\frakp(\disc(1, y, \dots, y^{n-1})) - 2 \sum_{i=0}^{n-1} \floor{i
    \frac{\nu_\frakp(D)}{n}}. \] Now \[ \nu_\frakp(\disc(\calO_\frakp')) = \sum_{\frakP \mid \frakp}
    d(\frakP | \frakp) f(\frakP | \frakp) = n - \frac{d}{n} n = n - d, \] whence it suffices to show
    that $\nu_\frakp(\disc(z_0, \dots, z_n)) = n - d$. First, note that \[ \nu_\frakp(\disc(1, y,
    \dots, y^{n-1})) = \nu_\frakp(D) (n - 1), \] whence we have to show that \[ \nu_\frakp(D) (n -
    1) - 2 \sum_{i=0}^{n-1} \floor{i \frac{\nu_\frakp(D)}{n}} = n - \gcd(n, \nu_\frakp(D)). \]
    Define $s := \nu_\frakp(D)$, then this simplifies to $s (n - 1) - 2 \sum_{i=0}^{n-1} \floor{
    \frac{i s}{n}} = n - \gcd(n, s)$. But this is actually shown in
    \cite[Proposition~3.1]{wu-radicalffsintbases}.
  \end{proof}
  
  We now want to construct elements $f_\frakP$, for $\frakP \mid \frakp$, which satisfy
  \begin{equation}
    \tag{$\ast$}
    \nu_\frakP(f_\frakP) = 1 \qquad \text{and} \qquad \forall \frakP' \mid \frakp, \frakP' \neq
    \frakP : \nu_{\frakP'}(f_\frakP) = 0.
  \end{equation}
  The ring $\calO_\frakp'$ is a principal ideal domain whose non-zero prime ideals correspond to the
  places~$\frakP \mid \frakp$, where $\frakP$ corresponds to the prime ideal $\frakm_\frakP \cap
  \calO_\frakp'$. Any generator of $\frakm_\frakP \cap \calO_\frakp'$ satisfies ($\ast$), and vice
  versa, any element~$f_\frakP$ satisfying ($\ast$) is a generator for $\frakm_\frakP \cap
  \calO_\frakp'$. Hence, these elements~$f_\frakP$ allow us to describe the non-zero prime
  ideals~$\frakm_\frakP \cap \calO_\frakp'$ in an elegant way.
  
  This will be used in Sections~\ref{matrixrepueip} and \ref{comprr} to directly write down an
  $\frako_\infty$-basis for the $\calO_\infty'$-ideal~$\fraka$ with $\nu_\frakP(\fraka) = t_\frakP$,
  when integers $t_\frakP \in \Z$, $\frakP \mid \frakp$ are given.
  
  \begin{proposition}
    \label{placedecomposition}
    Let $\frakp \in \PP_{k(x)}$ and $d := \gcd(n, \nu_\frakp(D))$.
    \begin{enuma}
      \item If $d = 1$, let $a, b \in \Z$ with $a \nu_\frakp(D) + b n = 1$. Then $z := y^a \pi^b$
      satisfies $\nu_\frakP(z) = 1$ for the only place~$\frakP \in \PP_K$ lying above $\frakp$.
      \item If $d > 1$, let $a, b \in \Z$ such that $a \frac{n}{d} + b \frac{\nu_\frakp(D)}{d} =
      1$. Let $\pi$ be a uniformizer for $\frakp$ and set $\pi' := \pi^a y^b$. If $d = n$, we can
      choose $\pi' = \pi$.
      
      Let $P : \frako_\frakp \to \kappa(\frakp)$ be the projection and let $\alpha := P(D
      \pi^{-\nu_\frakp(D)})$. Let the factorization of $Z^d - \alpha$ over $\kappa(\frakp)$ be
      $\prod_{i=1}^t g_i$ with pairwise distinct monic prime polynomials~$g_i \in
      \kappa(\frakp)[Z]$. Let $\hat{g}_i \in \frako_\frakp[Z]$ be monic polynomials with
      $P(\hat{g}_i) = g_i$, $1 \le i \le t$; in case $\deg \frakp = 1$, we can choose $\hat{g}_i =
      g_i$.
      
      If $\fraka_i := \pi' \calO_\frakp' + g_i(y \pi^{-\nu_\frakp(D) / d}) \calO_\frakp'$, then
      $\fraka_1, \dots, \fraka_t$ are exactly the non-zero prime ideals of $\calO_\frakp'$. Let
      $\frakP_i \in \PP_K$ be the place corresponding to $\fraka_i$, i.e. $\frakm_{\frakP_i} \cap
      \calO_\frakp' = \fraka_i$.
      
      \begin{enumi}
        \item If $d = n$:
        
        For $z_i^{(1)} := g_i(y \pi^{-\nu_\frakp(D) / d})$ and $z_i^{(2)} := z_i^{(1)} + \pi$, we
        have $\nu_{\frakP_j}(z_i^{(s)}) = 0$ for $j \neq i$, $s = 1, 2$ (or $z_i^{(1)} = 0$), and we
        have $\nu_{\frakP_i}(z_i^{(s)}) > 0$ for both $s$, and $\nu_{\frakP_i}(z_i^{(s)}) = 1$ for
        at least one $s$.
        
        Moreover, $\nu_{\frakP_i}(z_i^{(1)}) > 1$ is the case if, and only if, $\pi \not\in
        z_i^{(1)} \calO_\frakp'$. This is the case if, and only if,
        $\nu_\frakp(\Norm_{K/k(x)}(z_i^{(1)})) > 1$.
        \item If $d < n$:
        
        For $z_i := g_i(y \pi^{-\nu_\frakp(D) / d}) + \pi'$, we have $\nu_{\frakP_j}(z_i) = 0$ for
        $j \neq i$ and $\nu_{\frakP_i}(z_i) = 1$. In particular, $\fraka_i = z_i \calO_\frakp'$.
      \end{enumi}
    \end{enuma}
  \end{proposition}
  
  \begin{proof}
    The case~(a) was already shown in the previous proposition. Now, consider~(b). First note that
    $\pi'$ is a uniformizer for every place~$\frakP$ of $K$ lying above $\frakp$: $\nu_\frakP(\pi')
    = a \nu_\frakP(\pi) + b \nu_\frakP(y) = a \frac{n}{d} + b \frac{1}{d} \nu_\frakp(D) = 1$.
    
    Let $\rho := y^{n/d}$ and $K' := k(x)(\rho)$; then the minimal polynomial of $\rho$ over $k(x)$
    is $Y^d - D$, and the minimal polynomial of $y$ over $K'$ is $Y^{n/d} - \rho$. Note that if
    $\hat{\frakP}$ is a place of $K'$ lying above $\frakp$, then $e(\hat{\frakP} | \frakp) = 1$ as
    $d \mid \nu_\frakp(D)$. Thus, $\nu_\frakP(\rho) = \frac{1}{d} \nu_\frakP(D) = \frac{1}{d}
    \nu_\frakp(D)$ is coprime to $\frac{n}{d}$: this shows that the extension $K / K'$ with $K =
    K'(y)$ of degree~$n/d$ is totally ramified in $\hat{\frakP}$. Moreover, the extension $K' /
    k(x)$ with $K' = k(x)(\rho)$ is unramified in $\frakp$. (In case $d = n$, $K' = K$.)
    
    Now $z := \rho \pi^{-\nu_\frakp(D) / d}$ satisfies $\calO_\frakp \cap K' = \frako_\frakp[z]$ by
    part~(b) of the previous proposition. Let $P : \frako_\frakp \to \kappa(\frakp)$ be the
    projection and $\alpha := P(D \pi^{-\nu_\frakp(D)})$; as $\nu_\frakp(D \pi^{-\nu_\frakp(D)}) =
    0$ we have $\alpha \neq 0$. Now $(Z^d - \alpha)' = d Z^{d - 1}$ is coprime to $Z^d - \alpha$,
    whence $Z^d - \alpha$ is squarefree. Let the factorization of $Z^d - \alpha$ over
    $\kappa(\frakp)$ be $\prod_{i=1}^t g_i$ with pairwise distinct monic prime polynomials~$g_i \in
    \kappa(\frakp)[Z]$. Let $\hat{g}_i \in \frako_\frakp[Z]$ be monic polynomials with $P(\hat{g}_i)
    = g_i$, $1 \le i \le t$.
    
    By Kummer's Theorem \cite[p.~76, Theorem~III.3.7]{stichtenoth}, the places $\hat{\frakP} \in
    \PP_{K'}$ lying above $\frakp$ correspond to the $g_i$. Let $\hat{\frakP}_i$ be the place
    corresponding to $g_i$. Now
    \begin{equation}\tag{$\ast$}
      \frakm_{\hat{\frakP}_i} \cap (\calO_\frakp' \cap K') = \pi
      (\calO_\frakp' \cap K') + \hat{g}_i(z) (\calO_\frakp' \cap K').
    \end{equation}
    
    If $\hat{g}_i(z) = 0$, we must have $\deg \hat{g}_i = d$, whence $\hat{\frakP}_i$ is the only
    place lying over $\frakp$. As $K/K'$ is unramified, there is exactly one place of $K$ lying
    above $\frakp$. In this case, we get that $f_i = \pi'$ satisfies the condition. Hence, we assume
    that $\hat{g}_i(z) \neq 0$ for all $i$.
    
    By ($\ast$) we see that $\nu_{\frakP_j}(\hat{g}_i(z)) = 0$ for $j \neq i$ and
    $\nu_{\frakP_i}(\hat{g}_i(z)) > 0$. As $\nu_{\hat{\frakP}_i}(\pi) = 1$, we have
    $\nu_{\frakP_i}(\hat{g}_i(z) + \pi) = 1$ in case $\nu_{\frakP_i}(\hat{g}_i(z)) > 1$. Moreover,
    $\pi \in \hat{g}_i(z) (\calO_\frakp' \cap K')$ if, and only if, $1 = \nu_{\frakP_i}(\pi) \ge
    \nu_{\frakP_i}(\hat{g}_i(z))$. As $\nu_{\frakP_j}(\hat{g}_i(z)) = 0 < \nu_{\frakP_j}(\pi)$, we
    also have $\nu_{\frakP_j}(\hat{g}_i(z) + \pi) = 0$ for $j \neq i$. Finally,
    $\nu_\frakp(\Norm_{K/k(x)}(\hat{g}_i(z)) = \sum_{\frakP \mid \frakp} \nu_\frakP(z) =
    \nu_{\frakP_i}(z)$. This shows (b)~(i), i.e. the case $n = d$.
    
    Now assume $d < n$. In that case, $\nu_{\frakP_i}(\hat{g}_i(z)) = \frac{n}{d}
    \nu_{\hat{\frakP}_i}(\hat{g}_i(z)) \ge \frac{n}{d} > 1$, whence $\nu_{\frakP_i}(\hat{g}_i(z) +
    \pi') = 1$. Moreover, as before, $\nu_{\frakP_j}(\hat{g}_i(z) + \pi') = 0$ for $j \neq
    i$. Finally, note that this implies $(\hat{g}_i(z) + \pi') \calO_\frakp' = \frakm_{\frakP_i}
    \cap \calO_\frakp'$. And as $\pi'$ and $\hat{g}_i(z)$ clearly lie in this ideal, we have
    $\fraka_i = \hat{g}_i(z) \calO_\frakp' + \pi' \calO_\frakp'$.
  \end{proof}
  
  We have seen how to, given a place~$\frakp \in \PP_{k(x)}$,
  \begin{itemize}
    \item decide whether an easy monogenic basis for $\calO'_\frakp$ exists and, if yes, how to
    obtain it;
    \item find an easy to describe integral basis for $\calO'_\frakp$;
    \item find all places~$\frakP \in \PP_K$ lying above $\frakp$; and
    \item find generators of the non-zero prime ideals $\frakm_\frakP \cap \calO'_\frakp$ of
    $\calO'_\frakp$.
  \end{itemize}
  To compute these, we need to compute
  \begin{enuma}
    \item greatest common divisors of two integers and the corresponding B\'ezout identities,
    \item a uniformizer for a place~$\frakp \in \PP_{k(x)}$,
    \item the projection $P_\frakp : \frako_\frakp \to \kappa(\frakp)$ for a specific element of
    $\frako_\frakp$,
    \item the factorization of a polynomial of the form~$Y^n - \alpha$ in some $\kappa(\frakp)[Y]$.
  \end{enuma}
  The first can be done using the Extended Euclidean Algorithm (see
  \cite{vzgathen-moderncomputeralgebra}). For (b) and (c), distinguish between the infinite
  place~$\frakp_\infty$ and the finite places.
  
  For $\frakp = \frakp_\infty$, we have $\kappa(\frakp) \cong k$, and a uniformizer is given by $\pi
  = \frac{1}{x}$. If $f = \lambda \frac{g}{h} \in k(x)^*$ with $g, h \in k[x]$ monic and $\lambda
  \in k^*$, then \[ P_\frakp(f) = \begin{cases} 0 & \text{if } \deg g < \deg h, \\ \lambda &
    \text{if } \deg g = \deg h \end{cases}. \] Now assume that $\frakp$ is a finite place,
  corresponding to a monic irreducible polynomial~$p \in k[x]$. Then we can choose $\pi = p$, and we
  have$\kappa(\frakp) \cong k[x]/(p)$ and $P_\frakp(f) = f + (p)$. In particular, if $p = x -
  \lambda$, then $\kappa(\frakp) \cong k$ and $P(f) = f(\lambda)$.
  
  This allows us to describe the residue field~$\kappa(\frakp)$ and compute the residue map
  $P_\frakp$ for all places of $k(x)$.
  
  We are mainly interested in the case that $k = \F_q$ is a finite field of $q$ elements; in this
  case, $\kappa(\frakp)$ is a finite field of $q^{\deg \frakp}$ elements. In that case,
  factorization of polynomials is well understood
  \cite[Section~14]{vzgathen-moderncomputeralgebra}. The special case of radical polynomials~$Y^n -
  \alpha$ is even easier, if one does not need to know the exact factorization but only the number
  of degree~$d$ prime divisors for all $d \ge 1$. We will investigate this in the next section.
  
  \section{On the Factorization of $Y^n - \alpha$ over a Finite Field}
  \label{factorization}
  
  
  In the following, we are interested in obtaining information on the factorization of $Y^n -
  \alpha$ over a finite field~$k = \F_q$, where $n$ is coprime to $q$ and $\alpha \neq 0$. These
  assumptions imply that $Y^n - \alpha$ is squarefree and that all roots are non-zero. We will give
  an explicit algorithm (Algorithm~\ref{splittingalgorithm}) which computes the degrees of the
  irreducible factors of $Y^n - \alpha$ over $\F_q$, and bound its running time.
  
  First, we are interested in the roots of $Y^n - \alpha$ in a specific extension $\F_{q^m}$ of
  $\F_q$, $m \ge 1$. The in $\F_{q^m}$ are exactly the roots of $\gcd(Y^n - a, Y^{q^m - 1} - 1)$.
  By taking a generator~$\beta$ of $\F_{q^m}^*$ and solving the Discrete Logarithm Problem~$\beta^x
  = \alpha$, it is possible to reduce $Y^n - \alpha$ to a linear equation modulo~$q^m - 1$. The same
  can also be achieved by explicitly computing $\gcd(Y^n - \alpha, Y^{q^m - 1} - 1)$, which can be
  done completely without polynomial arithmetic:
  
  \begin{lemma}
    \label{Y^n-a_polygcdlemma}
    Let $k$ be any field and $\alpha, \beta \in k \setminus \{ 0 \}$. Consider $Y^n - \alpha$ and
    $Y^m - \beta$ with $n, m \in \N_{>0}$. Write $\gcd(n, m) = \lambda n + \mu m$ with $\lambda, \mu
    \in \Z$. Then \[ \gcd(Y^n - \alpha, Y^n - \beta) = \begin{cases} Y^{\gcd(n, m)} - \alpha^\lambda
      \beta^\mu & \text{if } \alpha^{\frac{m}{\gcd(n, m)}} = \beta^{\frac{n}{\gcd(n, m)}}, \\ 1 &
      \text{else-wise.} \end{cases} \] It can be computed using the following algorithm:
    \begin{algorithm}{Computing $\gcd(Y^n - \alpha, Y^m - \beta)$}
      \begin{enum1}
        \item Set $A := \Matrix{ m & n \\ 0 & 1 \\ 1 & 0 } = (a_{ij})_{ij}$.
        \item While $a_{11} \neq 0$, do:
        \begin{enumi}
          \item Compute $q := a_{12} \mymod a_{11}$ (so that $0 \le q < a_{11}$).
          \item Set $B := \Matrix{ -q_n & 1 \\ 1 & 0 }$.
          \item Set $A := A \cdot B$.
        \end{enumi}
        \item Compute $c := \alpha^{a_{21}} \beta^{a_{31}}$.
        \item If $c \neq 1$, return $1$.
        \item Set $\gamma := \alpha^{a_{22}} \beta^{a_{32}}$ and return $Y^{a_{12}} - \gamma$.
      \end{enum1}
    \end{algorithm}
  \end{lemma}
  
  \begin{proof}
    One obtains this by following the Euclidean Algorithm applied on $f$ and $g$. For that one has
    to investigate what the long division of $Y^n - \alpha$ by $Y^m - \beta$ does. Write $n = q m +
    r$ with $0 \le r < m$; as $(Y^m)^q - \beta^q$ is divisible by $Y^m - \beta$, we can write $Y^{m
    q} = \beta^q + h(Y) (Y^m - \beta)$ with $h \in k[Y]$. Then $Y^n - \alpha = (Y^m)^q Y^r - \alpha
    = h(Y) Y^r (Y^m - \beta) + \beta^q Y^r - \alpha$, whence \[ (Y^n - \alpha) \mymod (Y^m - \beta) =
    \beta^q Y^r - \alpha = \beta^q (Y^r - \alpha \beta^{-q}). \]
  \end{proof}

  This allows us to give a precise answer about the number of zeroes of $Y^n - \alpha$ in
  $\F_{q^m}$:
  
  \begin{corollary}
    \label{rootcheckcorollary}
    Let $m \in \N$ and let $d = \gcd(q^m - 1, n)$. Then $Y^n - \alpha \in \F_q[Y]$ has zeroes in
    $\F_{q^m}$ if, and only if, $\alpha^{\frac{q^m - 1}{d}} = 1$. If that is the case, it has
    precisely $d$ zeroes.
  \end{corollary}
  
  \begin{proof}
    The number of zeroes of $Y^n - \alpha$ in $\F_{q^m}$ is given by $\deg \gcd(Y^{q^m - 1} - 1, Y^n
    - \alpha)$. The degree is $> 0$ by the previous lemma if, and only if, $1^{\frac{n}{d}}
    \alpha^{-\frac{q^m - 1}{d}} = 1$, which is the case if, and only if, $\alpha^{\frac{q^m - 1}{d}}
    = 1$. If this is the case, the degree of $\gcd(Y^{q^m - 1} - 1, Y^n - \alpha)$ is $\gcd(q^m - 1,
    n)$ by the previous lemma.
  \end{proof}
  
  Note that this can be proven alternatively, without the use of Lemma~\ref{Y^n-a_polygcdlemma}:
  
  \begin{proof}[Alternative proof of Corollary~\ref{rootcheckcorollary}.]
    Write $\F_{q^m}^* = \ggen{\beta}$ for a primitive element~$\beta$. Write $\alpha = \beta^x$ with
    $x \in \N$. Now $\beta^y$ is an $n$-th root of $\alpha$ if, and only if, $n y \equiv x \pmod{q^m
    - 1}$. But this is known to be solvable if, and only if, $d := \gcd(q^m - 1, n)$ divides $x$; in
    that case, there exist~$d$ solutions.
  \end{proof}
  
  Moreover, we can determine the splitting field of $Y^n - \alpha$:
  
  \begin{corollary}
    \label{splittingdegree}
    The degree of the splitting field of $Y^n - \alpha$ over $\F_q$ is given by \[ \order_{\Z / n
    \order_{\F_q^*}(\alpha) \Z}(q). \]
  \end{corollary}
  
  \begin{proof}
    Let $m \in \N_{>0}$. Then the condition that $Y^n - \alpha$ splits over $\F_{q^m}$ is equivalent
    to $\gcd(q^m - 1, n) = n$ and $\alpha^{\frac{q^m - 1}{\gcd(q^m - 1, n)}} = 1$. This is easily
    seen to be equivalent to $n \order_{\F_q^*}(\alpha) \mid (q^m - 1)$.
  \end{proof}
  
  Note that there is a field-theoretic interpretation: in case a field $L$ contains all solutions of
  $Y^n - \alpha$, $\alpha \neq 0$, it must contain a primitive $n$-th root of unity. Now $\F_{q^m}$
  contains such a root if, and only if, $n \mid (q^m - 1)$ as $\F_{q^m}^*$ is cyclic of order~$q^m -
  1$. Finally, a field $L$ containing a primitive $n$-th root of unity contains either none or all
  roots of $Y^n - \alpha$.
  
  We now want to compute the degrees of the irreducible polynomials appearing in the factorization
  of $Y^n - \alpha$ over $\F_q$. For that, it suffices to determine the number~$n_m$ of roots of
  $Y^n - \alpha$ in $\F_{q^m}$ which do not lie in any subfield, $m \ge 1$. Then the number of
  irreducible factors of degree~$m$ is given by $\frac{n_m}{m}$.
  
  Hence, this can be done with the following algorithm:
  
  \begin{algorithm}{Compute the degrees of the factorization of $Y^n - \alpha$ over $\F_q$}
    \label{splittingalgorithm}
    \alginput{$n \in \N_{>0}$, a prime power~$q$, $\alpha \in \F_q$.}
    \algoutput{a list $(a_1, \dots, a_m)$ such that $a_i$ is the number of irreducible factors of
    $Y^n - a$ of degree~$i$.}
    \begin{enum1}
      \item Compute $m'' = \order_{\F_q^*}(\alpha)$, $m' = \order_{(\Z / n m'' \Z)^*}(q)$ and set $m
      = \min\{ m', n \}$.
      \item Set $a_1 = \dots = a_m = 0$.
      \item For $i = 1, \dots, m$, do:
      \begin{enumi}
        \item If $i \nmid m'$, continue with the next $i$.
        \item Compute $d = \gcd(q^i - 1, n)$.
        \item If $m''$ does not divide $\frac{q^i - 1}{d}$, continue with the next $i$.
        \item For $j = 2, \dots, \floor{\frac{m}{i}}$, do:
        \begin{enuma}
          \item Set $a_{i \cdot j} = a_{i \cdot j} - (a_i + d)$.
        \end{enuma}
        \item Set $a_i = \frac{a_i + d}{i}$.
      \end{enumi}
      \item Return $(a_1, \dots, a_m)$.
    \end{enum1}
  \end{algorithm}
  
  \begin{proposition}
    \label{splittingalgorithmruntime}
    The algorithm computes the degrees of the factorization of $Y^n - \alpha$ over $\F_q$ in \[
    \calO(n \log^3 n + n \log n \cdot \log^2 q + \log^3 q \cdot \log \log q) \] binary
    operations. We assume that the factorizations of $q - 1$ and $n$ are given, and the
    factorization of $p - 1$ for every prime~$p$ dividing $(q - 1) n$.
  \end{proposition}
  
  Note that the time required to factor $n$ and all $p - 1$ for $p \mid n$ is subexponential in
  $\log n$ for each of this numbers. Hence, the time required for this is less than $\calO(n \log
  n)$, i.e. it is negligible compared to the total running time of the algorithm.
  
  \begin{proof}
    By Corollary~\ref{splittingdegree}, $m'$ is the degree of the splitting field of $Y^n - \alpha$
    over $\F_q$. Hence, the degrees of all possible irreducible factors of $Y^n - \alpha$ divide
    $m'$. Moreover, the maximal degree of such a polynomial is bounded by $m$.
    
    Now let $i$ be a divisor of $m''$ which is $\le m$, and let $d = \gcd(q^i - 1, n)$. Then, by
    Corollary~\ref{rootcheckcorollary}, $Y^n - \alpha$ has roots in $\F_{q^i}$ if, and only if,
    $\alpha^{\frac{q^i - 1}{d}} = 1$, i.e. if, and only if, the order of $\alpha$ divides $\frac{q^i
    - 1}{d}$. In case it has roots in $\F_{q^i}$, the number of roots is $d$, again by
    Corollary~\ref{rootcheckcorollary}. Now we have to subtract from $d$ the number of roots
    already lying in subfields of $\F_{q^i}$ or, alternatively, one subtracts the roots lying in
    subfields from $a_i$ when their number is determined; the latter is done in the loop in
    Step~(3\;iv). Finally, one divides the number by $i$ as $\F_{q^i} / \F_q$ is Galois and the
    minimal polynomials of the roots in $\F_{q^i}$ which lie in no subfield have degree~$i$. This
    shows that the algorithm is correct.
    
    Now $m'' = \calO(q)$, $m' = \calO(n q)$ and $m = \calO(n)$. Computing $\gcd(q^i - 1, n)$ can be
    done by first evaluating $q^i - 1$ modulo $n$ and then computing the $\gcd$, whence this
    requires $\calO(\log i \cdot \log^2 n + \log^2 n) \subseteq \calO(\log i \log^2 n) \subseteq
    \calO(\log^3 n)$ binary operations. Moreover, to check whether $m'' \mid \frac{q^i - 1}{d}$ it
    suffices to evaluate $q^i - 1$ modulo $d m'' = \calO(\log(n q))$, which can be done in
    $\calO(\log i \log^2 (n q)) \subseteq \calO(\log^3 n + \log n \cdot \log^2 q)$ binary
    operations. Hence, the running time of the loop in Step~(3) is $\calO(n \cdot \log n \cdot
    (\log^2 n + \log^2 q))$ binary operations.
    
    As we know the factorization of $q - 1$, we can compute $m'' = \order_{\F_q^*}(\alpha)$ using a
    fast order computation in $\calO(\frac{\log^3 q \cdot \log \log q}{\log \log \log q})$ binary
    operations; see \cite[p.~117, Proposition~7.3]{sutherland-phd}. This algorithm will give the
    factorization of $m''$, whence we know the factorization of $n m''$ and can use that information
    to compute the factorization of $\phi(n m'')$. Hence, the order of $q$ in $(\Z/n m''\Z)^*$ can
    be computed in $\calO(\frac{\log^3 (n q) \cdot \log \log (n q) }{\log \log \log (n q)})$ binary
    operations as $n m'' = \calO(n q)$. In particular, Step~(1) requires $\calO(\frac{\log^3 (n q)
    \cdot \log \log (n q) }{\log \log \log (n q)}) \subseteq \calO(\log^3 n \cdot \log \log n +
    \log^3 q \cdot \log \log q)$ binary operations.
    
    This gives the stated total running time.
  \end{proof}
  
  This section shows how to compute the degrees of the prime factors in the factorization of $Y^n -
  \alpha$ over a finite field. We have seen in Proposition~\ref{monogenicintbases} that this allows
  us to describe the degrees of the places of $K$ lying above a place~$\frakp \in \PP_{k(x)}$. This
  completely suffices to determine whether $K$ has an infinite place of degree~one or to compute the
  Euler factor for $\frakp$ in Section~\ref{eulerprodapprox}. It does not suffice, though, to
  describe the places lying above $\frakp$ as in Proposition~\ref{placedecomposition}; for that, we
  need to compute the prime factors in the decomposition.
  
  To factor $Y^n - \alpha$, one can first compute the number~$n_d$ of prime divisors of degree~$d$
  for all $d \ge 1$ as sketched above; as a side result, this algorithm can compute $\gcd(Y^n -
  \alpha, Y^{q^d - 1} - 1)$ as well and use the same technique to eliminate all prime divisors of
  degree~$< d$. Then the resulting polynomial~$h_d \in \F_q[Y]$ is of degree~$n_d d$, and is the
  product of $n_d$ distinct prime factors of degree~$d$, to which, for example, the
  Cantor-Zassenhaus algorithm can be applied (see \cite{vzgathen-moderncomputeralgebra}).
  
  \section{Integral Bases, Part~2}
  \label{intbasespart2}
  
  In this section, we want to generalize a result of Q.~Wu \cite{wu-radicalffsintbases} on how to
  give an explicit $k[x]$-basis of $\calO$, the integral closure of $k[x]$ in $K$. We will need this
  for computing Riemann-Roch spaces in the next sections.
  
  Moreover, we will give an inequality for the degree of a certain rational function appearing in
  the integral basis and give a precise statement when equality happens; this will be important for
  the determination of the exact constant field of $K / k$ in Section~\ref{compexactconstfield}.
  
  Define \[ z := y \prod_{i=-\infty}^\infty f_i^{-\floor{i / n}}; \] then $k(x, y) = k(x, z)$ and \[
  z^n = \sgn(D) \prod_{i=-\infty}^\infty f_i^{i \mymod n} = \sgn(D) \prod_{j=1}^{n-1} \biggl(
  \prod_{i=-\infty}^\infty f_{j + i n} \biggr)^j =: \tilde{D} \in k[x]. \] For $j = 1, \dots, n-1$,
  define \[ \tilde{G}_i := \prod_{i=-\infty}^\infty f_{j + i n} \in k[x]; \] then $\tilde{G}_1,
  \dots, \tilde{G}_{n-1}$ are pairwise coprime, monic and squarefree polynomials such that \[ z^n =
  \sgn(D) \prod_{i=1}^{n-1} \tilde{G}_i^i = \tilde{D}. \]
  If we set $\tilde{D}_i := \prod_{j=1}^{n-1} \tilde{G}_j^{\floor{\frac{i j}{n}}}$ for $j \in \{ 0,
  \dots, n - 1 \}$, then \[ \tilde{D}_i = \prod_{j=1}^{n-1} \biggl( \prod_{k=-\infty}^\infty f_{j +
  n k} \biggr)^{\floor{\frac{i j}{n}}} = \prod_{j=-\infty}^\infty f_j^{\floor{\frac{i (j \mymod
  n)}{n}}}. \] By \cite{wu-radicalffsintbases}, $z^i / \tilde{D}_i$, $i = 0, \dots, n - 1$ is a
  $k[x]$-basis of its integral closure in $K$. Now \[ \frac{z^i}{\tilde{D}_i} = y^i
  \prod_{j=-\infty}^\infty f_j^{-\floor{\frac{i (j \mymod n)}{n}} - i \floor{\frac{j}{n}}} =
  \frac{y^i}{D_i} \] with \[ D_i := \prod_{j=-\infty}^\infty f_j^{\floor{\frac{i (j \mymod n)}{n}} +
  i \floor{\frac{j}{n}}} = \prod_{j=-\infty}^\infty f_j^{\floor{\frac{i j}{n}}}. \] In particular,
  $D_0 = 1$. Note that we no longer have $D_1 = 1$ in general; this only happens if $f_i = 0$ for $i
  < 0$ and for $i \ge n$. Hence, we have proven:
  
  \begin{theorem}
    \label{finiteintbasis}
    Let $K = K(x, y)$ where $y^n = D$ with $D \in k(x)^*$. If $D = \sgn(D) \prod_{i=-\infty}^\infty
    f_i^i$ is the squarefree decomposition of $D$ and if \[ D_i := \prod_{j=-\infty}^\infty
    f_j^{\floor{\frac{i j}{n}}}, \] then $D_0 = 1$ and \[ 1, \frac{y}{D_1}, \frac{y^2}{D_2}, \dots,
    \frac{y^{n-1}}{D_{n-1}} \] is an integral basis for $\calO$, the integral closure of $k[x]$ in
    $K$. \qed
  \end{theorem}
  
  We will now compare the degree of $D_i$ to $\ceil{i \frac{\deg D}{n}}$, which will later allow us
  to make statements on the exact constant field of $K$.
  
  \begin{lemma}
    \label{Dibounds}
    Let $i \in \{ 1, \dots, n - 1 \}$.
    \begin{enuma}
      \item We have $\deg D_i \le \floor{i \frac{\deg D}{n}} \le \ceil{i \frac{\deg D}{n}}$.
      \item We have $\deg D_i = \ceil{i \frac{\deg D}{n}}$ if, and only if, \[ \forall j \in \Z :
      \deg f_j = 0 \vee n \mid i j. \]
    \end{enuma}
  \end{lemma}
  
  \begin{proof}
    First, as $\floor{a} + \floor{b} \le \floor{a + b}$ for all $a, b \in \R$, note that
    \begin{align*}
      \deg D_i ={} & \sum_{j=-\infty}^{\infty} \floor{\frac{i j}{n}} \deg f_j \le
      \sum_{j=-\infty}^\infty \floor{\frac{i j}{n} \deg f_j} \\
      {}\le{} & \floor{\sum_{j=-\infty}^{\infty} \frac{i j}{n} \deg f_j} = \floor{i \frac{\deg
      D}{n}} \le \ceil{i \frac{\deg D}{n}}.
    \end{align*}
    This shows (a), and moreover it shows that (b) is equivalent to that the following three
    conditions are satisfied simultaneously:
    \begin{enum1}
      \item $\forall j \in \Z : \floor{\frac{i j}{n}} \deg f_j = \floor{\frac{i
      j}{n} \deg f_j}$;
      \item $\sum_{j=-\infty}^\infty \floor{\frac{i j}{n} \deg f_j} = \floor{\sum_{j=-\infty}^\infty
      \frac{i j}{n} \deg f_j}$; and
      \item $\floor{\frac{i}{n} \deg D} = \ceil{\frac{i}{n} \deg D}$.
    \end{enum1}
    Now (3) is clearly equivalent to
    \begin{equation}\tag{3'}
      n \mid i \deg D = \sum_{j=1}^{n-1} i j \deg f_j.
    \end{equation}
    To attack (1) and (2), note that $\floor{a} + \floor{b} = \floor{a + b}$ if, and only if, $\{ a
    \} + \{ b \} < 1$; here, $\{ a \} := a - \floor{a}$ is the fractional part of $a$. Then, (1) and
    (2) are equivalent to (1') and (2'), respectively:
    \begin{align}
      & \forall j \in \Z : \{ \tfrac{i j}{n} \} \deg f_j < 1 \tag{1'} \\
      & \sum_{j=-\infty}^\infty \{ \tfrac{i j}{n} \deg f_j \} < 1 \tag{2'}
    \end{align}
    Now $\{ \frac{a}{n} \} = \frac{a \mymod n}{n}$, whence these conditions can be rewritten as
    \begin{align}
      & \forall j \in \Z : (i j \mymod n) \deg f_j < n \tag{1''} \\
      & \sum_{j=-\infty}^\infty (i j \deg f_j \mymod n) < n \tag{2''}
    \end{align}
    Note that (3') is the case if, and only if, $n \mid \sum_{j=-\infty}^\infty (i j \deg f_j \mymod
    n)$. Therefore, (3') and (2'') are together equivalent to
    \begin{align}
      & \sum_{j=-\infty}^\infty (i j \deg f_j \mymod n) = 0, \tag{4}
    \end{align}
    which is clearly equivalent to
    \begin{align}
      & \forall j \in \Z : n \mid (i j \deg f_j). \tag{4'}
    \end{align}
    Thus, $\deg D_i = \ceil{i \frac{\deg D}{n}}$ is equivalent to (1'') and (4'), i.e. to
    \begin{equation}\tag{5}
      \forall j \in \Z : (i j \mymod n) \deg f_j < n \wedge n \mid (i j \deg f_j).
    \end{equation}
    If $\deg f_j = 0$ or $n \mid i j$, we clearly have $(i j \mymod n) \deg f_j < n \wedge n \mid (i
    j \deg f_j)$. Hence, assume that $\deg f_j > 0$ and $n \nmid i j$. In case $n \mid (i j \deg
    f_j)$, we have $\frac{n}{\gcd(n, \deg f_j)} \mid i j$, whence $i j \mymod n$ is a multiple of
    $\frac{n}{\gcd(n, \deg f_j)}$. But $i j \mymod n \neq 0$ as $n \nmid i j$, whence $i j \mymod n
    \ge \frac{n}{\gcd(n, \deg f_j)}$. But then, \[ (i j \mymod n) \deg f_j \ge n \frac{\deg
    f_j}{\gcd(n, \deg f_j)} \ge n. \] Therefore, (5) is equivalent to $\forall j \in \Z : \deg f_j =
    0 \vee n \mid i j$, what we wanted to show.
  \end{proof}
  
  Hence, we saw how to obtain a $k[x]$-basis of the integral closure~$\calO$ of $k[x]$ in the
  radical function field $K$, which is easy to write down once one has the squarefree decomposition
  of $D(x)^*$. This basis is of the form $\frac{y^i}{D_i}$ with $0 \le i < n$, i.e. it allows to
  efficiently test an element~$f = \sum_{i=0}^{n-1} a_i y^i \in K$ with $a_i \in k(x)$ for being
  integral: this is the case if, and only if, $D_i a_i \in k[x]$ for $0 \le i < n$. Moreover, we
  need this result to describe the size of certain transformation matrices in Section~\ref{comprr},
  as well as finding formulas for the degree of the exact constant field over $k$ and the genus of
  $K / k$, which only depend on the numerical data $(n, (\deg f_i)_{i \in \Z})$ (see
  Sections~\ref{compexactconstfield} and \ref{diffdeggenus}).
  
  \section{Matrix Representation of Uniformizing Elements for the Infinite Places}
  \label{matrixrepueip}
  
  This section prepares work for the next section. We want to find a matrix representing certain
  uniformizing elements for the infinite places and make statements on the size of the entries of
  these matrices and their inverses. The final, quantitative result is given in
  Proposition~\ref{MiMiinv-matrixbounds}.
  
  Let $\frakp_1, \dots, \frakp_s$ be the infinite places of $K / k(x)$. Then there exist
  elements~$h_i \in \calO_\infty$ with $\nu_{\frakp_i}(h_i) = 1$ and
  $\nu_{\frakp_j}(h_i) = 0$ for $j \neq i$, as described in Proposition~\ref{placedecomposition}.
  
  We have seen that $\hat{w}_0, \dots, \hat{w}_{n-1}$ with $\hat{w}_i = y^i x^{\floor{i \frac{-\deg
  D}{n}}}$ is a $\frako_\infty$-basis of $\calO_\infty$, and at the same time a $k(x)$-basis of
  $K$. Hence, we can represent $h_i$ and $h_i^{-1}$ as $n \times n$-matrices $M_i$ and $M_i^{-1}$
  over $k(x)$ with respect to this basis, by interpreting them as $k(x)$-vector space endomorphisms
  of $K$ given by \[ v \mapsto h_i v \qquad \text{and} \qquad v \mapsto h_i^{-1} v \qquad (v \in
  K). \] We are interested to give bounds on the numerators and denominators of these matrices. More
  precisely, given a matrix~$M = (m_{ij})_{ij} \in k(x)^{n \times n}$, we define $\deg M =
  \max_{i,j} \deg m_{ij}$.  The denominator of $M$ is a monic polynomial~$d \in k[x]$ of minimal
  degree which satisfies~$d M \subseteq k[x]^{n \times n}$, and the numerator of $M$ is $d M$;
  hence, we are interested in upper bounds for $\deg d$ and $\deg (d M)$.
  
  Note that $\deg(A B) \le \deg A + \deg B$ and $\deg(A + B) \le \max\{ \deg A, \deg B \}$ for all
  $A, B \in k(x)^{n \times n}$. Moreover, note that the strict triangle inequality does not hold in
  general, and that $\deg$ is \emph{not} multiplicatively as soon as $n > 1$.
  
  We first begin with a small lemma on B\'ezout identities:
  
  \begin{lemma}
    \label{bezouteqnlemma}
    Let $a, b \in \N_{>0}$ such that $d := \gcd(a, b) > 0$ satisfies $d < \min\{ a, b \}$.
    \begin{enumi}
      \item There exist $\lambda$, $\mu$ with $0 \le \lambda < \frac{b}{d}$ and $0 \le \mu <
      \frac{a}{d}$ such that $d = \lambda a - \mu b$.
      \item There exist $\lambda'$, $\mu'$ with $0 \le \lambda' < \frac{b}{d}$ and $0 \le \mu' <
      \frac{a}{d}$ such that $d = -\lambda' a + \mu' b$.
    \end{enumi}
  \end{lemma}
  
  \begin{proof}
    Let $\lambda'', \mu'' \in \Z$ be arbitrary with $d = \lambda'' a + \mu'' b$. Then the set of all
    solutions of $d = \alpha a + \beta b$ is \[ \{ (\lambda'' + x \tfrac{b}{d}, \mu'' - x
    \tfrac{a}{d}) \mid x \in \Z \}. \]
    \begin{enumi}
      \item First assume $d < b$. Choose $x$ such that $0 \le \lambda := \lambda'' + x \tfrac{b}{d}
      < \tfrac{b}{d}$. Note that there is exactly one such $x$. Then \[ \mu := -(\mu'' - x
      \tfrac{a}{d}) = \frac{\lambda a - d}{b}, \] whence \[ -\frac{d}{b} \le \mu < \frac{a b/d -
      d}{b} = \frac{a}{d} - \frac{d}{b}. \] As $0 < \frac{d}{b} < 1$ we get $0 \le \mu <
      \frac{a}{d}$, as we wanted.
      
      In case $d = b$, we have $a > b$ and $d = 1 \cdot a - (\frac{a}{b} - 1) \cdot b$. Then $0 \le
      \mu := 1 \le 1 = \frac{b}{d}$ and $0 \le \lambda := \frac{a}{b} - 1 < \frac{a}{b} =
      \frac{a}{d}$.
      
      \item This follows from (i) by switching $a$ and $b$. \qedhere
    \end{enumi}
  \end{proof}
  
  Write $D = \sgn(D) \frac{D_{num}}{D_{denom}}$ with pairwise coprime monic polynomials~$D_{num},
  D_{denom} \in k[x]$, and set $D_{\deg} := \max\{ \deg D_{num}, \deg D_{denom} \}$; note that
  $D_{\deg}$ is the \emph{height} of $D \in k(x)$. Note that \[ D_{num} = \prod_{i=1}^\infty f_i^i
  \in k[x] \qquad \text{and} \qquad D_{denom} = \prod_{i=1}^\infty f_{-i}^i \in k[x]. \]
  
  \begin{lemma}
    \label{multmatrixestimate}
    Let $f = \sum_{i=0}^{n-1} a_i x^{b_i} y^i \in K$ with $a_i \in k$, $b_i \in \Z$. Define a matrix
    $M = (m_{ij})_{ij} \in k(x)^{n \times n}$ such that $f \hat{w}_j = \sum_{i=0}^{n-1} m_{ij}
    \hat{w}_i$. If $d \in k[x]$ is monic and of minimal degree with $d M \in k[x]^{n \times n}$,
    then
    \begin{align*}
      \deg d \le{} & \deg D_{denom} + \max\bigl\{ \ceil{-b_{\min} - \tfrac{1}{n}} + 1 + \max\{ \deg
      D, 0 \}, 0 \bigr\} \\
      {}={} & \max\{ \floor{-b_{\min} - \tfrac{1}{n}} + 1 + D_{\deg}, \deg D_{denom} \} =: A
    \end{align*}
    and \[ \deg(d M) \le A + \floor{b_{\max} - \tfrac{1}{n}} + 1 + \max\{ \deg D, 0 \}; \] here,
    \begin{align*}
      b_{\min} :={} & \min\{ b_i + \tfrac{i}{n} \deg D \mid a_i \neq 0 \} \\ \text{and} \qquad
      b_{\max} :={} & \max\{ b_i + \tfrac{i}{n} \deg D \mid a_i \neq 0 \}.
    \end{align*}
  \end{lemma}
  
  \begin{proof}
    Note that $\hat{w}_i = y^i x^{\floor{i \frac{-\deg D}{n}}} = y^i x^{-\ceil{i \frac{\deg
    D}{n}}}$. Therefore, \[ a_i x^{b_i} y^i \cdot \hat{w}_i = x^{b_i - \ceil{j \frac{\deg D}{n}}}
    y^{i + j}. \] Next, note that
    \begin{align*}
      & - \ceil{a \tfrac{\deg D}{n}} + \ceil{b \tfrac{\deg D}{n}} = \floor{ -a \tfrac{\deg D}{n} } -
      \floor{ -b \tfrac{\deg D}{n}} \\
      {}={} & (b - a) \tfrac{\deg D}{n} + \tfrac{1}{n} \left( ((-b \deg D) \mymod n) - ((-a \deg D)
      \mymod n) \right),
    \end{align*}
    and the last term lies in the interval $[-1 + \frac{1}{n}, 1 - \frac{1}{n}]$.
    
    Now we want to estimate $\deg d m_{ij}$ as well as $\deg d$ itself. For that, we distinguish two
    cases. Both show that $d$ must be of the form $x^\ell D_{denom}$ for some $\ell \in \Z$, $\ell
    \ge -\nu_0(D_{denom})$ as soon as $f \not\in k$, and allow to give conditions on $\ell$; here,
    $\nu_0(D_{denom})$ denotes the exact power of $x$ dividing $D_{denom}$.
    \begin{enum1}
      \item The first case is $i + j < n$. In that case, \[ m_{i+j,j} = a_i x^{b_i - \ceil{ j
      \tfrac{\deg D}{n}} + \ceil{(i + j) \tfrac{\deg D}{n}}}. \] In case $a_i \neq 0$, we have \[
      \ell \ge -b_i + \ceil{ j \tfrac{\deg D}{n}} - \ceil{(i + j) \tfrac{\deg D}{n}} \] and
      \begin{align*}
        & -b_i + \ceil{ j \tfrac{\deg D}{n}} - \ceil{(i + j) \tfrac{\deg D}{n}} \\
        {}\le{} & -b_i - \tfrac{i}{n} \deg D + 1 - \tfrac{1}{n} \le -b_{\min} + 1 - \tfrac{1}{n}.
      \end{align*}
      Moreover, \[ \deg(x^\ell m_{i+j,j}) \le \ell + b_i + \tfrac{i}{n} \deg D + 1 - \tfrac{1}{n}
      \le b_{\max} + \ell + 1 - \tfrac{1}{n}. \]
      \item The second case is $i + j \ge n$. In that case, \[ m_{i+j-n,j} = a_i x^{b_i - \ceil{ j
      \tfrac{\deg D}{n}} + \ceil{(i + j - n) \tfrac{\deg D}{n}}} \frac{D_{num}}{D_{denom}}. \] In
      case $a_i \neq 0$, we have \[ \ell \ge -b_i + \ceil{ j \tfrac{\deg D}{n}} - \ceil{(i + j - n)
      \tfrac{\deg D}{n}} \] and
      \begin{align*}
        & -b_i + \ceil{ j \tfrac{\deg D}{n}} - \ceil{(i + j - n) \tfrac{\deg D}{n}} \\
        {}\le{} & -b_i - \tfrac{i}{n} \deg D + \deg D + 1 - \tfrac{1}{n} \le -b_{\min} + \deg D + 1
        - \tfrac{1}{n}.
      \end{align*}
      Moreover, 
      \begin{align*}
        \deg(x^\ell m_{i+j-n,j}) \le{} & \ell + b_i + \tfrac{i}{n} \deg D + \deg D + 1 -
        \tfrac{1}{n} \\
        {}\le{} & \ell + b_{\max} + \deg D + 1 - \tfrac{1}{n}.
      \end{align*}
    \end{enum1}
    This shows that
    \begin{align*}
      \deg d ={} & \ell + \deg D_{denom} \\
      {}\le{} & \deg D_{denom} + \max\bigl\{ \ceil{-b_{\min} + 1 - \tfrac{1}{n} + \max\{ \deg D, 0
      \}}, 0 \bigr\} = A
    \end{align*}
    and \[ \deg(x^\ell D_{denom} m_{i,j}) \le A + \floor{b_{\max} + 1 - \tfrac{1}{n} + \max\{ \deg
    D, 0 \}} \] as $\deg(x^\ell D_{denom} m_{i,j}) \in \Z$.
  \end{proof}
  
  Now let us consider $M_i$ and $M_i^{-1}$ obtained from choosing~$h_i$ as in
  Proposition~\ref{placedecomposition}. Remember that we have $\hat{w}_i = x^{\floor{-\frac{i \deg
  D}{n}}} y^i$. We distinguish between three cases:
  \begin{enum1}
    \item $\gcd(\deg D, n) = 1$, i.e. $K$ has exactly one infinite place. First, assume $\deg D >
    0$. Then, by Lemma~\ref{bezouteqnlemma}~(ii) there exist $s, t \in \Z$ with $1 = -s \deg D + t
    n$ with $0 < s < n$ and $0 < t < \deg D$ and we can choose~$h_1 := y^s x^{-t}$.
    
    Next, assume $\deg D < 0$. Then, by Lemma~\ref{bezouteqnlemma}~(i) there exist $s, t \in \Z$
    with $1 = s (-\deg D) + t n$ with $0 < s < n$ and $\deg D < t < 0$ and we can choose~$h_1 := y^s
    x^{-t}$.
    \item $\gcd(\deg D, n) = n$, i.e. the infinite places are unramified. Let $\pi : \frako_\infty
    \to k$ be the projection and $\alpha := \pi(D x^{-\deg D})$. Write $Y^n - \alpha = \prod_{i=1}^t
    g_i$ with pairwise coprime monic polynomials~$g_i \in k[Y]$. Then we can choose $h_i = g_i(y
    x^{-\deg D / n})$ or $h_i = g_i(y x^{-\deg D / n}) + x^{-1}$. Note that $\deg g_i = \deg P_i <
    n$.
    
    \item $\gcd(\deg D, n) =: d \in \{ 2, \dots, n - 1 \}$, in which case there are up to $d$
    infinite places which are all ramified. Let $\pi : \frako_\infty \to k$ be the projection and
    $\alpha := \pi(D x^{-\deg D})$. Write $Y^d - \alpha = \prod_{i=1}^t g_i$ with pairwise coprime
    monic polynomials~$g_i \in k[Y]$.
    
    Next, in case $\deg D > 0$, write $1 = -s \frac{\deg D}{d} + t \frac{n}{d}$ with $s, t \in \Z$
    and $0 < s < \frac{n}{d}$ and $0 < t < \frac{\deg D}{d}$; then $\nu_{P_i}(y^s x^{-t}) = 1$. In
    case $\deg D < 0$, write $1 = s \frac{-\deg D}{d} + t \frac{n}{d}$ with $s, t \in \Z$ and $0 < s
    < \frac{n}{d}$ and $\frac{\deg D}{d} < t < 0$; then $\nu_{P_i}(y^s x^{-t}) = 1$.
    
    In both cases, we can choose $h_i = g_i(y^{n / d} x^{-\deg D / d}) + y^s x^{-t}$. Note that
    $\deg g_i = \deg P_i \le d$.
  \end{enum1}
  
  Now let $d_i \in k[x]$ be monic and of minimal degree with $d_i M_i \in k[x]^{n \times n}$. Using
  Lemma~\ref{multmatrixestimate}, we can give upper bounds for $\deg d_i$ and $\deg (d_i M_i)$. We
  are only interested in quantitative results, but note that one can work out sharper bounds in
  detail using the above material. All involved $\calO$-constants do not depend on $n$ or $D$. We
  have the same three cases as above:
  \begin{enum1}
    \item Note that in this case, $b_{\min} = b_{\max} = -t + \frac{s}{n} \deg D$, whence we have
    $\abs{b_{\min}} = \abs{b_{\max}} = \calO(\abs{\deg D})$.
    \item In this case, $b_{\min} = -1$ and $b_{\max} = 0$ (as $g_i(0) \neq 0$). 
    \item Here, we have $1 - (\frac{1}{d} - \frac{1}{n}) \abs{\deg D} \le b_{\min} \le b_{\max} \le
    -1 + (\frac{1}{d} - \frac{1}{n}) \abs{\deg D}$ (as $g_i(0) \neq 0$). Therefore, $\abs{b_{\min}},
    \abs{b_{\max}} = \calO(D_{\deg})$.
  \end{enum1}
  Applying the lemma, we obtain \[ \deg d_i = \calO(D_{\deg}) \qquad \text{and} \qquad \deg(d_i M_i)
  = \calO(D_{\deg}) \] in all cases.
  
  Instead of repeating the same investigation for the inverses~$M_i^{-1}$, we use some results from
  Linear Algebra. For a matrix~$M \in R^{n \times n}$, where $R$ is any commutative unitary ring,
  one can define the adjugate matrix~$\adj(M) \in R^{n \times n}$ which satisfies $\adj(M) M = M
  \adj(M) = \det(M) \cdot I_n$, where $I_n$ is the $n \times n$ identity matrix. The elements of
  $\adj(M)$ are cofactors of $M$, i.e.\ up to sign determinants of $(n - 1) \times (n -
  1)$~submatrices of $M$. Hence, if $M \in k[x]^{n \times n}$ and we consider the Leibniz formula
  for the determinant, we get $\deg \adj(M) \le (n - 1) \deg M$. Therefore, if $M \in k(x)^{n \times
  n}$ and $d \in k[x] \setminus \{ 0 \}$ is monic and of minimal degree with $d M \in k[x]^{n \times
  n}$, then \[ M^{-1} = d (d M)^{-1} = d \det(d M)^{-1} \adj(d M) = \frac{1}{d^{n - 1} \det M}
  \adj(d M). \] Hence, if $d' \in k[x] \setminus \{ 0 \}$ is monic and of minimal degree with $d'
  M^{-1} \subseteq k[x]^{n \times n}$, we have \[ \deg d' \le (n - 1) \deg d + \max\{ \deg \det M, 0
  \} \] and \[ \deg (d' M^{-1}) \le \max\{ -\deg \det M, 0 \} + (n - 1) \deg (d M). \] We can use
  this to show our quantitative result:
  
  \begin{proposition}
    \label{MiMiinv-matrixbounds}
    Now, if $M = M_i^{t_i}$ for some $i$ and $t_i \in \Z$, and $d \in k[x] \setminus \{ 0 \}$ is
    monic and of minimal degree with $d M \in k[x]^{n \times n}$, then \[ \deg d = \calO(\abs{t_i} n
    D_{\max}) \qquad \text{and} \qquad \deg (d M) = \calO(\abs{t_i} n D_{\max}). \] In case $t_i \ge
    0$, we get the stronger result \[ \deg d = \calO(t_i D_{\max}) \qquad \text{and} \qquad \deg (d
    M) = \calO(t_i D_{\max}). \]
  \end{proposition}
  
  \begin{proof}
    Let $d_i' \in k[x] \setminus \{ 0 \}$ be monic and of minimal degree such that $d_i' M_i^{-1}
    \in k[x]^{n \times n}$. Note that $\det M_i = \Norm_{K/k(x)}(h_i)$ equals the norm of $h_i$,
    whence $\nu_{\frakp_\infty}(\det M_i) = \deg P_i$, i.e.\ $\deg \det M_i = -\deg P_i$. Using the
    above results, we see that \[ \deg d_i' = \calO(n D_{\max}) \qquad \text{and} \qquad \deg d_i'
    M_i^{-1} = \calO(n D_{\max}). \]
    
    First, assume that $t_i \ge 0$. Then $M = M_i^{t_i}$, and we can choose~$d = d_i^{t_i}$. Hence,
    $\deg d M = t_i \deg d_i M_i = \calO(t_i D_{\max})$ and $\deg d = \calO(t_i D_{\max})$ from the
    above discussion. Next, assume that $t_i < 0$. In that case, $M = (M_i^{-1})^{-t_i}$, and we can
    choose~$d = (d_i')^{-t_i}$. Hence, $\deg d M = (-t_i) \deg d_i' M_i^{-1} = \calO((-t_i) n
    D_{\max})$ and $\deg d = \calO((-t_i) n D_{\max})$.
  \end{proof}
  
  This shows that the matrices we can use to represent infinite places can be given using $n^2 + 1$
  polynomials whose degree is in $\calO(n D_{\max})$. If we have any selection of integers~$t_1,
  \dots, t_s \in \Z$, then $\prod_{i=1}^s M_i^{t_i}$ can be written in the form $M/d$ with $M \in
  k[x]^{n \times n}$ and $d \in k[x]$ such that all involved polynomials are of
  degree~$\calO(D_{\deg} n \sum_{i=1}^s \abs{t_i})$. We use this to show a bound on the running time
  of Riemann-Roch space computations in the next section.
  
  \section{Computation of Riemann-Roch Spaces}
  \label{comprr}
  
  This section is very central in this paper: it combines everything prepared so far to show how to
  compute Riemann-Roch spaces in radical function fields over finite constant fields, and gives an
  estimate on the running time. Part of the material from this section will be used in the next
  section to compute the exact constant field of $K/k(x)$. The algorithm we will use in this section
  is based on an algorithm of F.~He\ss\ \cite{hessRR}, and is also described in
  \cite{ff-tioagfoaur}. The main results in this section are given in
  Proposition~\ref{preciseRRruntime} and its corollary.
  
  Assume that the infinite places are $\frakp_1, \dots, \frakp_s$. Let $\fraka$ be a non-zero
  $\calO$-ideal and $t_i \in \Z$, $1 \le i \le s$. We are interested in computing a $k$-basis of \[
  B(\fraka, (t_1, \dots, t_s)) := \{ f \in \fraka \mid \nu_{\frakp_i}(f) \ge -t_i, 1 \le i \le s
  \}. \] If $\fraka = \prod_{\frakp \nmid \frakp_\infty} (\frakm_\frakp \cap \calO)^{n_\frakp}$ is
  the factorization of $\fraka$ into prime ideals of $\calO$, then $B(\fraka, (t_1, \dots, t_s))$ is
  exactly the Riemann-Roch space \[ L\left( -\sum_{\frakp \nmid \frakp_\infty} n_\frakp \frakp +
  \sum_{i=1}^s t_i \frakp_i \right). \] Note that any divisor of $K$ can be represented in such a
  form; also see \cite[Section~2.5]{diem-habil}.
  
  Now $\hat{v}_i := y^i / D_i$, $0 \le i < n$ is an integral basis for $\calO$ by
  Theorem~\ref{finiteintbasis}. Next, a $\frako_\infty$-basis for $\calO_\infty$ is given by
  $\hat{w}_i := x^{\floor{i \frac{-\deg D}{n}}} y^i$, $0 \le i < n$ by
  Proposition~\ref{monogenicintbases}. Now, if $M \in k(x)^{n \times n}$ satisfies $(\hat{v}_0,
  \dots, \hat{v}_{n-1}) = (\hat{w}_0, \dots, \hat{w}_{n-1}) M$, then $M = (m_{ij})_{ij}$ with
  $m_{ij} = 0$ for $i \neq j$, and $m_{ii} = x^{\ceil{i \frac{\deg D}{n}}} / D_i$.
  
  Write \[ \fraka = \frac{1}{d(\fraka)} \sum_{i=0}^{n-1} \left( \sum_{j=0}^i a_{ij} \hat{v}_i
  \right) k[x] \] with $\deg a_{ij} < \deg a_{ii}$ for $j < i$, the $a_{ii}$'s being monic, and
  $(\fraka)$ monic of minimal degree. Set $T_\fraka := (\frac{1}{d(\fraka)} a_{ij})_{ij}$ with
  $a_{ij} = 0$ for $j > i$; then $v_0, \dots, v_{n-1}$ is a $k[x]$-basis of $\fraka$ if we set
  $(v_0, \dots, v_{n-1}) := (\hat{v}_0, \dots, \hat{v}_{n-1}) T_\fraka = (\hat{w}_0, \dots,
  \hat{w}_{n-1}) M T_\fraka$. If $M_\fraka := M T_\fraka = (m'_{ij})_{ij}$, then \[ m_{ij}' = m_{ii}
  \frac{a_{ij}}{d(\fraka)} = \frac{a_{ij} x^{\ceil{i \frac{\deg D}{n}}}}{D_i d(\fraka)}. \]
  
  Next, we need elements~$h_i \in \calO_\infty$ with $\nu_{\frakp_i}(h_i) = 1$ and
  $\nu_{\frakp_j}(h_i) = 0$ for $j \neq i$. Then, we can define $M_i$ by $(h_i \hat{w}_0, \dots, h_i
  \hat{w}_{n-1}) = (\hat{w}_0, \dots, \hat{w}_{n-1}) M_i$ and $M((t_i)_i) := \prod_{i=1}^t
  M_i^{t_i}$. We already discussed how to find such elements in Proposition~\ref{placedecomposition}
  and the previous section. Now we can apply the Reduction Algorithm of Paulus \cite{paulusReduct}
  to the matrix~$M((t_i)_i) M_\fraka$. There exists two matrices $U \in \Gl_n(\frako_\infty)$ and $V
  \in \Gl_n(k[x])$ such that \[ U M((t_i)_i) M_\fraka V = \Matrix{ x^{\lambda_1} & 0 & \cdots & 0 \\
  0 & \ddots & \ddots & \vdots \\ \vdots & \ddots & \ddots & 0 \\ 0 & \cdots & 0 & x^{\lambda_n}
  } \] with $\lambda_1, \dots, \lambda_n \in \Z$. Actually, the algorithm computes $M((t_i)_i)
  M_\fraka V$, the $\lambda_i$ can be recovered as the maximal degree appearing in the $i$-th column
  of that matrix, and the algorithm can easily be modified to explicitly compute $V$ as well without
  affecting the asymptotic running time. Now, if we compute $(\tilde{v}_0, \dots, \tilde{v}_{n-1})
  = (v_0, \dots, v_{n-1}) V = (\hat{v}_1, \dots, \hat{v}_n) T_\fraka V$, then a $k$-basis of
  $B(\fraka, (t_1, \dots, t_s))$ is given by \[ \{ x^j \tilde{v}_i \mid 0 \le i < n, \; 0 \le j
  -\lambda_i \}. \]
  
  In the rest of the section, we are interested in estimating the running time. Recall that we
  defined $\deg A = \max_{i,j} \deg{a_{ij}}$ for a matrix~$A = (a_{ij})_{ij} \in k(x)^{n \times
  n}$. The running time of Paulus' algorithm, applied to a matrix $A \in k[x]^{n \times n}$, is
  $\calO(n^4 (\deg A)^2)$ operations in $k$; see \cite[Proposition~3.3]{paulusReduct}.
  
  First, let us write $M_\fraka = \frac{1}{d} M_\fraka'$ with $M_\fraka' \in k[x]^{n \times n}$ and
  $d \in k[x]$. For that, note that $\frac{D_{num}}{D_i} = \prod_{j=-\infty}^{-1}
  f_j^{-\floor{\frac{i j}{n}}} \prod_{j=1}^\infty f_j^{-\floor{\frac{i j}{n}} + j} \in k[x]$; this
  shows that we can choose $d = D_{num} d(\fraka)$. Set $M_\fraka' := d M_\fraka \in k[x]^{n \times
  n}$. Then
  \begin{align*}
    \deg M_\fraka' ={} & \max_{i,j=0, \dots, n-1} \left( \deg a_{ij} + \ceil{\frac{i \deg D}{n}} -
    \deg D_i + \deg D_{num} \right) \\
    {}\le{} & \max_{i,j=0,\dots,n-1} \deg a_{ij} + \max\{ \deg D, 0 \} + \deg D_{num} \\
    {}={} & \calO(\deg d(\fraka) T_\fraka + D_{\deg})
  \end{align*}
  and \[ \deg d = \deg d(\fraka) + \deg D_{num}. \]
  
  Next, using Proposition~\ref{MiMiinv-matrixbounds} we see that if $d' \in k[x]$ is monic and of
  minimal degree with $d' M((t_i)_i) \in k[x]^{n \times n}$, then \[ \deg d' = \calO\biggl(
  \sum_{i=1}^s \abs{t_i} n D_{\deg} \biggr) \quad \text{and} \quad \deg(d' M((t_i)_i)) =
  \calO\biggl( \sum_{i=1}^s \abs{t_i} n D_{\deg} \biggr). \]
  
  \begin{proposition}
    \label{preciseRRruntime}
    Let $\fraka$ be an $\calO$-ideal of $K$ represented with respect to the $k[x]$-basis $y^0/D_0,
    \dots, y^{n-1}/D_{n-1}$ of $\calO$ as $M/d$ with $M \in k[x]^{n \times n}$ and $d \in
    k[x]$. Moreover, let $t_1, \dots, t_s \in \Z$ be integers. Then the running time required for
    computing a $k$-basis of \[ L\biggl( \sum_{\frakp \nmid \frakp_\infty} n_\frakp \frakp +
    \sum_{i=1}^s t_i \frakp_i \biggr) \] is \[ \calO\biggl( \biggl(n + \sum_{i=1,\dots,s} \max\{ 0,
    \log \abs{t_i} \}\biggr) n^3 \biggl( \biggl(1 + n \sum_{i=1}^s \abs{t_i}\biggr) D_{\deg} + \deg
    M \biggr)^2 \biggr) \] operations in $k$, and requires a storage of \[ \calO\biggl( n^2 \biggl(
    \biggl(1 + n \sum_{i=1}^s \abs{t_i}\biggr) D_{\deg} + \deg M \biggr) \biggr) \] elements of $k$.
  \end{proposition}
  
  \begin{proof}
    We assume that the result is given in a `compact form', i.e. $\tilde{v}_i$ is only given once
    and not for every power of $x$ multiplied to it. Moreover, each element in the result has to be
    divided by $d(\fraka)$.
    
    Note that the given running time and space requirements essentially describe the running time of
    Paulus' algorithm and the matrix~$d d' M((t_i)_i) M_\fraka$. Clearly, the storage required for
    matrix multiplications is at most a constant multiple of the storage for one matrix.
    
    For computation of a $k$-basis, we also have to construct the matrix $M((t_i)_i) M_\fraka$ first
    by multiplying all required matrices together, and then, after applying Paulus' algorithm,
    collecting the information from the algorithm to compose the $k$-basis of the Riemann-Roch
    space.
    
    Note that multiplying two matrices~$A, B \in k[x]^{n \times n}$ requires $\calO(n^3 (\deg A +
    \deg B)^2 )$ operations in $k$. Hence, we have to show that the number of matrix multiplications
    is $\calO(n + \sum_{i=1,\dots,s} \max\{ 0, \log \abs{t_i} \})$.
    
    The last step requires multiplication of the matrix~$V$ obtained from Paulus' algorithm with the
    basis~$(v_0, \dots, v_{n-1}) = (\hat{v}_0, \dots, \hat{v}_{n-1}) M$. The matrix obtained from
    Paulus' algorithm is of the same size as the input matrix, i.e. the entries are of degree \[
    \calO\biggl( \biggl(1 + n \sum_{i=1}^s \abs{t_i}\biggr) D_{\deg} + \deg M \biggr). \]
    Multiplying it by $M$ and by the elements $\hat{v}_i$ shows that the result is of the same
    size. Here, two matrix multiplications are required.
    
    Finally, for computing $M_\fraka$, every entry of $M$ is multiplied with a polynomial; the
    running time is less than the running time for one matrix multiplication. For computation of
    $M((t_i)_i)$, one requires $\calO( \sum_{i=1}^s (\max\{ \log \abs{t_i}, 0 \} + 1) )$~matrix
    multiplications, and multiplying $M((t_i)_i)$ with $M_\fraka$ requires a last one. As $s \le n$,
    the claim follows.
  \end{proof}
  
  Noting that $\sum_{i=1}^s \log \max\{ \abs{t_i}, 1 \} \le s \log \max\{ \sum_{i=1}^s \abs{t_i}, 1
  \}$ and $s \le n$, we get the following special case:
  
  \begin{corollary}
    In case there exists some~$G > 0$ with $\deg M = \calO(n G)$ and $\sum_{i=1}^s \abs{t_i} =
    \calO(G)$, then the running time is \[ \calO\bigl( n^6 D_{\deg}^2 G^2 \log G \bigr) \]
    operations in $k$ and the storage requirement is \[ \calO\bigl( n^3 G D_{\deg} \bigr) \]
    elements of $k$. \qed
  \end{corollary}
  
  Note that in case $\fraka$ is a product of at most two reduced ideals and the sum of the
  $\abs{t_i}$'s is $\calO(g)$, where $g$ is the genus of $K / k$, we see that we can choose $G = g$
  (see \cite{ff-tioagfoaur}). In particular, the running time required for a giant step or a
  reduction in the sense of \cite{ff-tioagfoaur} is \[ \calO( n^6 D_{\deg}^2 g^2 (\log g)^2 ) \]
  operations in $k$. In Corollary~\ref{genusbigODdeg} we will see that $g = \calO(n D_{\deg})$,
  whence we obtain the running time \[ \calO( n^8 D_{\deg}^4 (\log n + \log D_{\deg})^2 ). \] This
  is a much more precise estimate than the standard estimates as in \cite{diem-habil} that
  arithmetic is polynomial in $n$, $g$ and the size of the representation of $K$; the latter is in
  this case bounded by $2 D_{\deg}$.
  
  \section{Computing the Exact Constant Field}
  \label{compexactconstfield}
  
  In this section we will give an explicit description of the exact constant field~$k'$ of $K /
  k(x)$, using the methods from the previous section on the computation of Riemann-Roch spaces. Note
  that $k' = L(0)$, the Riemann-Roch space of the zero divisor. For the main result, see
  Theorem~\ref{exconstfieldthm}.
  
  To compute $L(0)$, consider the matrix $M = (m_{ij})_{ij}$ with $m_{ij} = 0$ for $i \neq j$ and
  $m_{ii} = x^{\ceil{i \frac{\deg D}{n}}} / D_i$. Applying Paulus' algorithm \cite{paulusReduct} to
  this matrix will return the matrix itself. Let $\lambda_i := \deg m_{ii}$. Hence, a $k$-basis of
  $L(0)$ is given by $\{ \hat{v}_i x^j \mid 0 \le j < -\lambda_i, 0 \le i < n \}$. Obviously,
  $-\lambda_i = \deg D_i - \ceil{i \frac{\deg D}{n}}$. Clearly, $-\lambda_0 = 0$, which is not
  surprising after all as $\hat{v}_0 = 1 \in L(0)$ and $x \not\in L(0)$.
  
  Therefore, we have \[ \dim_k L(0) = 1 + \sum_{i=1}^{n-1} \max\left\{ 1 + \deg D_i - \ceil{i
  \frac{\deg D}{n}}, 0 \right\} \] and a $k$-basis of $L(0)$ is given by $x^j \hat{v}_i$ with $0 \le
  i < n$ and $j = 0, \dots, \max\left\{ \deg D_i - \ceil{i \frac{\deg D}{n}}, -1 \right\}$. Now by
  Lemma~\ref{Dibounds}~(a), $\deg D_i - \ceil{i \frac{\deg D}{n}} \le 0$. Combining all this, we get
  the following result:
  
  \begin{proposition}
    We have \[ \dim_k L(0) = 1 + \abs{\biggl\{ i \;\biggm|\; 1 \le i < n, \; \deg D_i = \ceil{i
    \frac{\deg D}{n}} \biggr\}}, \] and a basis is given by \[ \biggl\{ 1, \frac{y^i}{D_i}
    \;\biggm|\; 1 \le i < n, \; \deg D_i = \ceil{i \frac{\deg D}{n}} \biggr\}. \] \qed
  \end{proposition}
  
  Our next aim is to describe the occurring integers~$i$ more precisely. Now
  Lemma~\ref{Dibounds}~(b) says that $\deg D_i < \ceil{i \frac{\deg D}{n}}$ if, and only if, there
  exists a $j \in \Z$ with $\deg f_j > 0$ and $n \nmid i j$. This condition can be described in an
  easier way:
  
  \begin{lemma}
    Let $n \in \N_{>0}$ and $S \subseteq \Z$ be a non-empty subset. Then, for a fixed~$i \in \Z$,
    \begin{equation}\tag{*}
      \exists j \in S : n \nmid i j
    \end{equation}
    is satisfied if, and only if, \[ n \nmid i \gcd(n, j \mid j \in S). \]
  \end{lemma}
  
  \begin{proof}
    Define $A := \bigcap_{j \in S} \frac{n}{\gcd(n, j)} \Z$. We first show that ($*$) is equivalent
    to $i \not\in A$. Note that $n \Z \subseteq A$.
    
    For that, assume ($*$). Then, by assumption, there exists a $j \in S$ with $n \nmid i j$, whence
    $\frac{n}{\gcd(n, j)} \nmid i$. But then, $i \not\in A$.
    
    Now, assume that $i \not\in A$. Then there exists at least one $j \in S$ with $i \not\in
    \frac{n}{\gcd(n, j)} \Z$, which means $\frac{n}{\gcd(n, j)} \nmid i$, i.e. $n \nmid i
    j$. Therefore, ($*$) holds.
    
    Thus, we have that ($*$) is equivalent to $i \not\in A$. Now let us study $A$. Clearly, \[ A =
    \bigcap_{j \in S} \frac{n}{\gcd(n, j)} \Z = \lcm\biggl( \frac{n}{\gcd(n, j)} \;\biggm|\; j \in S
    \biggr) \Z. \] Since $n$ is a common multiple of the $\frac{n}{\gcd(n, j)}$, $j \in S$, the
    $\lcm$ must be of the form $\frac{n}{\ell}$, $\ell \in \{ 1, \dots, n \}$. Now $\frac{n}{\gcd(n,
    j)} \mid \frac{n}{\ell}$ if, and only if, $\ell \mid \gcd(n, j)$. Therefore, \[ \lcm\biggl(
    \frac{n}{\gcd(n, j)} \;\biggm|\; j \in S \biggr) = \frac{n}{\gcd(\gcd(n, j) \mid j \in S)}. \]
    Moreover, note that $\gcd(\gcd(n, j) \mid j \in S) = \gcd(n, j \mid j \in S)$.  Summing up what
    we have so far, we get that ($*$) is equivalent to $i \not\in \frac{n}{\gcd(n, j \mid j \in S)}
    \Z$, i.e. to $n \nmid i \gcd(n, j \mid j \in S)$.
  \end{proof}
  
  Now we can give a precise statement on the exact constant field:
  
  \begin{theorem}
    \label{exconstfieldthm}
    We have
    \begin{align*}
      [k' : k] = \dim_k k' ={} & \abs{\{ i \in \{ 0, \dots, n - 1 \} \mid n \text{ divides } i
      \gcd(n, j \mid f_j \neq 1) \}} \\
      {}={} & \gcd(n, j \mid f_j \neq 1).
    \end{align*}
    Moreover, a $k$-basis of $k'$ is given by \[ \biggl\{ \frac{y^i}{D_i} \;\biggm|\; i \in \{ 0,
    \dots, n - 1 \}, \; \frac{n}{\gcd(n, j \mid f_j \neq 1)} \text { divides } i \biggr\}, \] where
    $\frac{y^0}{D_0} = 1$. In particular, $k$ is the exact constant field of $K$ if, and only if,
    $\gcd(n, j \mid f_j \neq 1) = 1$. \qed
  \end{theorem}
  
  This gives an easy to evaluate formula to decide whether $k$ is the exact constant field of $K /
  k$, and if not, to compute a $k$-basis of the exact constant field. In case $\gcd(n, j \mid f_j
  \neq 1) > 1$, $k' = k(y^i/D_i)$ with $i = \frac{n}{\gcd(n, j \mid f_j \neq 1)}$. Note that the
  minimal polynomial of $y^i/D_i$ over $k(x)$ is given by $T^{n/i} - \frac{D}{D_i^{n/i}}$; but since
  $[k' : k] = n/i$, it must as well be the minimal polynomial of $y^i/D_i$ over $k$, whence
  $\frac{D}{D_i^{n/i}} \in k$. This can be directly verified: since $f_j = 1$ for $j \not\in
  \frac{n}{i} \Z$, \[ D_i^{n/i} = \left( \prod_{j=-\infty}^\infty f_{\frac{n}{i} j}^{\floor{\frac{i
  \frac{n}{i} j}{n}}} \right)^{\frac{n}{i}} = \prod_{j=-\infty}^\infty f_{\frac{n}{i} j}^{j
  \frac{n}{i}} = \prod_{j=-\infty}^\infty f_j^j = \frac{D}{sgn(D)}. \] We see that $y^i/D_i =
  \sqrt[n/i]{sgn(D)}$, i.e. \[ k' = k(\sqrt[n/i]{sgn(D)}) = k(sgn(D)^{1/\gcd(n, j \mid f_j \neq
  1)}). \] This is also not very surprisingly, as $Y^n - D \in k(x)[Y]$ is not irreducibe over
  $k(sgn(D)^{1/\gcd(n, j \mid f_j \neq 1)})[Y]$ (compare Section~\ref{notation}); finally, $K /
  k'(x)$ is defined by $K = k'(x, y)$ with the relation $y^i = \hat{D}(x)$, where \[ \hat{D} :=
  sgn(D)^{i/n} \prod_{j=-\infty}^\infty f_{\frac{n}{i} j}^j = D(x)^{i/n} \in k'(x). \]
  
  \section{Computing the Degree of the Different and the Genus}
  \label{diffdeggenus}
  
  In this section we give two ways to compute the genus of $K / k(x)$: one method is to compute the
  degree of the different and using the Riemann-Hurwitz formula. The second method is more general
  applicable and is based on F.~He\ss' method on computing Riemann-Roch spaces: for the Riemann-Roch
  space algorithm, a special matrix has to be computed using an integral basis of $\calO$ and
  $\calO_\infty$. We show how one can extract the genus from this matrix using the Riemann-Roch
  theorem.
  
  Since all ramification is tame, the different of $K / k(x)$ and its degree can be computed using
  the ramification indices. For the ramification indices, we need the factorization of $D$, or at
  least we need to know the valuations and degrees of the appearing places. For that, recall that $D
  = \sgn(D) \prod_{i=-\infty}^\infty f_i^i$ with $\dots, f_{-2}, f_{-1}, f_1, f_2, \dots$ a sequence
  of squarefree, pairwise coprime elements of $k[x]$. Then, by \cite[p.~111,
  Proposition~III.7.3~(c)]{stichtenoth},
  \begin{align*}
    \deg \Diff(K / k(x)) ={} & \frac{n}{[k' : k]} \sum_{\frakp \in \PP_{k(x)}} \left(1 -
    \frac{\gcd(n, \nu_\frakp(D))}{n}\right) \deg \frakp \\
    {}={} & \frac{1}{[k' : k]} \sum_{i=-\infty}^\infty \left(n - \gcd(n, i)\right) \deg f_i,
  \end{align*}
  where $k'$ is the exact constant field of $K / k(x)$. The different itself can be computed in the
  same spirit, by factoring the $f_i$'s into a product of irreducible polynomials and determining
  the different exponents by the formula~$d(\frakP \mid \frakp) = \frac{n}{\gcd(n, \nu_\frakp(D))} -
  1$.
  
  Using the Hurwitz Genus Formula and the previous section, the genus~$g$ of $F$ equals
  \begin{align*}
    g ={} & 1 - \frac{n}{[k' : k]} + \tfrac{1}{2} \deg \Diff(K / k(x)) \\
    {}={} & 1 + \frac{1}{\gcd(n, j \mid f_j \neq 1)} \left( -n + \tfrac{1}{2}
    \sum_{i=-\infty}^\infty \left(n - \gcd(n, i)\right) \deg f_i \right).
  \end{align*}
  In particular, this shows:
  
  \begin{corollary}
    \label{genusbigODdeg}
    We have $g = \calO(n D_{\deg})$. \qed
  \end{corollary}
  
  A second way to compute the genus is at follows. It is essentially based on He\ss' idea on
  computing Riemann-Roch spaces together with the fact that for divisors~$D$ of large enough degree,
  $\dim_k L(D) = \deg D + [k' : k] (1 - g)$ by the Riemann-Roch theorem. This idea can be applied to
  any function field where a matrix~$M \in k(x)^{n \times n}$ is known which transforms a
  $\frako_\infty$-basis of $\calO_\infty$ into a $k[x]$-basis of $\calO$.
  
  Note that the algorithm of He\ss\ does not only computes a $k$-basis of $L(D)$ for some divisor
  $D$, but a $k$-basis of $L(D + t (x)_\infty)$ for \emph{any} $t \in \Z$; here, $(x)_\infty$
  denotes the pole divisor of $x$. If $t > 0$ is large enough, $\dim_k L( t (x)_\infty ) = \deg ( t
  (x)_\infty ) + [k' : k] (1 - g) = n t + [k' : k] (1 - g)$, whence \[ g = \frac{-\dim_k L(t
  (x)_\infty) + t n}{[k' : k]} + 1. \] Now, by the discussion in the previous section and
  \cite{hessRR}, \[ \dim_k L( t (x)_\infty ) = 1 + \sum_{i=1}^{n-1} \max\left\{ 1 + \deg D_i -
  \ceil{i \frac{\deg D}{n}} + t, 0 \right\}; \] hence, if $t$ is large enough, \[ \dim_k L( t
  (x)_\infty ) = 1 + n + \sum_{i=1}^{n-1} \biggl(\deg D_i - \ceil{i \frac{\deg D}{n}}\biggr) + t
  n, \] whence \[ g = 1 + \frac{- n - 1 + \sum_{i=1}^{n-1} \Bigl(\ceil{i \frac{\deg D}{n}} - \deg
  D_i\Bigr)}{\gcd(n, j \mid f_j \neq 1)}. \] Note that this method can be used for any function
  field, as long as integral bases of $\calO$ and $\calO_\infty$ are known. As one uses $D = 0$, one
  obtains $[k' : k] = \dim_k L(0)$, whence being able to compute integral bases and Riemann-Roch
  spaces suffices to compute $k'$, $[k' : k]$ and $g$.
  
  \section{Euler Product Approximation}
  \label{eulerprodapprox}
  
  In this section, we want to discuss Euler product approximation for radical function fields. The
  Euler product is another representation of the zeta function based on the places of $K$. We use
  the fact that the zeta function gives the $L$-polynomial of $K$, which in turn provides a way to
  compute the class number when evaluated at $t = 1$. In the following, we assume that $k = k' =
  \F_q$ is a finite field of $q$ elements, as well as the exact constant field of $K / k$; we have
  seen in the previous sections how to reduce to this case.
  
  We begin with giving the Euler product, divided by the Euler product representation of the zeta
  function of $k(x)$, via its factors. For $\frakp \in \PP_{k(x)}$, define \[ S(\frakp)(t) := (1 -
  t^{\deg \frakp}) \prod_{\frakP \mid \frakp} \frac{1}{1 - t^{\deg \frakP}}. \] We then have that \[
  \prod_{\frakp \in \PP_{k(x)}} S(\frakp)(t) = L_K(t) \in \Z[t] \] is the $L$-polynomial of $K$. It
  satisfies the functional equation $L_K(t) = q^g t^{2 g} L((q t)^{-1})$ and, more importantly, we
  have $\abs{\Pic^0(K)} = L_K(1)$. Note that we cannot evaluate $L_K(1)$ directly using the above
  product representation. But using the functional equation, we get $L_K(1) = q^g L(q^{-1})$, and
  $S(\frakp)(q^{-1})$ is well-defined for every~$\frakp \in \PP_{k(x)}$. Using the results from
  \cite{scheidlerstein-approxeulerprods}, we can determine the error if we only consider all
  places~$\frakp \in \PP_{k(x)}$ with $\deg \frakp \le \lambda$ in the product. For a $\lambda \in
  \N$ define $E_2(\lambda)$ by
  \begin{align*}
    \log E_2(\lambda) ={} & g \log q + \log \prod_{\stackrel{ \frakp \in \PP_{k(x)} }{ \deg \frakp
    \le \lambda}} S(\frakp)(q^{-1}) = g \log q - \sum_{\stackrel{ \frakp \in \PP_{k(x)} }{
    \deg \frakp \le \lambda}} (-\log S(\frakp)(q^{-1})) \\
    {}={} & g \log q - \sum_{\nu=1}^\lambda \sum_{\stackrel{ \frakp \in \PP_{k(x)} }{ \deg \frakp =
    \nu}} \biggl( \sum_{\frakP \mid \frakp} \log(1 - q^{-\deg \frakP}) - \log (1 - q^{-\deg \frakp})
    \biggr).
  \end{align*}
  In \cite[Theorem~4.2 and Theorem~4.3]{scheidlerstein-approxeulerprods}, bounds are given on
  $\abs{E_2(\lambda) - \abs{\Pic^0(K)}}$ which are of size $\calO(q^{g - (\lambda + 1) / 2})$, one
  of them being:
  \begin{theorem}[Scheidler--Stein \cite{scheidlerstein-approxeulerprods}]
    \label{eulerapproxthm}
    Let $k = \F_q$ be the exact constant field of $K / k$, and $n = [K : k(x)]$. We then have \[
    \abs{E_2(\lambda) - \abs{\Pic^0(K)}} \le E_2(\lambda) (e^{\psi_2(\lambda, \ell)} - 1), \] where
    $\ell$ is the smallest prime divisor of $\lambda + 1$ and
    \begin{align*}
      \psi_2(\lambda, \ell) ={} & \frac{2 g}{\lambda + 1} q^{-\frac{\lambda + 1}{2}} + \frac{n -
      1}{\lambda + 1} \frac{q}{q - 1} \frac{q^{\frac{\lambda + 1}{\ell}} - 1}{q^{\lambda + 1}} +
      \frac{n - 1}{\lambda + 1} q^{-(\lambda + 1)} \\
      {}+{} & \frac{2 g}{\lambda + 2} \frac{\sqrt{q}}{\sqrt{q} - 1} q^{-\frac{\lambda + 2}{2}} +
      \frac{2 (n - 1)}{\lambda + 2} \frac{q}{q - 1} \frac{q^{1 - \frac{1}{\ell}}}{q^{1 -
      \frac{1}{\ell}} - 1} q^{-(1 - \frac{1}{\ell}) (\lambda + 2)}.
    \end{align*}
  \end{theorem}
  
  \begin{proof}
    We have to show that our definition of $E_2(\lambda)$ coincides with the definition in
    \cite[Theorem~4.2]{scheidlerstein-approxeulerprods}; we denote their $E_2'(\lambda)$ by
    $\tilde{E}_2(\lambda)$. Then \[ \log \tilde{E}_2(\lambda) = A(K) + \sum_{m=1}^\lambda \frac{1}{m
    q^m} \sum_{\nu \mid m} \nu S_\nu(\tfrac{m}{\nu}) + \sum_{m=\lambda+1}^\infty \frac{1}{m q^m}
    \sum_{\nu \mid m \atop \nu \le \lambda} \nu S_\nu(\tfrac{m}{\nu}) \] with $A(K) = g \log q +
    \log S(\frakp_\infty)(q^{-1})$ and \[ S_\nu(i) := \sum_{\deg \frakp = \nu \atop \frakp \text{
    finite}} \sum_{j=1}^{n-1} z_j(\frakp)^i, \] where $z_j(\frakp)$ is defined by \[ \frac{1}{1 -
    t^{\deg \frakp}} S(\frakp)(t) = \prod_{\frakP \mid \frakp} \frac{1}{1 - t^{\deg \frakP}} =
    \frac{1}{1 - t^{\deg \frakp}} \prod_{j=1}^{n-1} \frac{1}{1 - z_j(\frakp) t^{\deg \frakp}}. \]
    Now
    \begin{align*}
      \frac{\tilde{E}_2(\lambda)}{\exp A(K)} ={} & \exp\left( \sum_{m=1}^\lambda \frac{1}{m q^m}
      \sum_{\nu \mid m} \nu S_\nu(\tfrac{m}{\nu}) + \sum_{m=\lambda+1}^\infty \frac{1}{m q^m}
      \sum_{\nu \mid m \atop \nu \le \lambda} \nu S_\nu(\tfrac{m}{\nu}) \right) \\
      {}={} & \prod_{m=1}^\infty \prod_{\nu \mid m \atop \nu \le \lambda}
      \exp(S_\nu(\tfrac{m}{\nu}))^{\frac{\nu}{m q^m}} = \prod_{m=1}^\infty \prod_{\nu \mid m \atop
      \nu \le \lambda} \prod_{\deg \frakp = \nu \atop \frakp \text{ finite}}
      \prod_{j=1}^{n-1} \exp(z_j(\frakp)^{m/\nu})^{\frac{\nu}{m q^m}} \\
      {}={} & \prod_{\deg \frakp \le \lambda \atop \frakp \text{ finite}} \prod_{j=1}^{n-1}
      \exp\left( \sum_{m=1}^\infty \frac{1}{m} \left( \frac{z_j(\frakp)}{q^{\deg \frakp}}
      \right)^m \right) \\
      {}={} & \prod_{\deg \frakp \le \lambda \atop \frakp \text{ finite}} \prod_{j=1}^{n-1}
      \exp\left( -\log\left(1 - \frac{z_j(\frakp)}{q^{\deg \frakp}} \right) \right) =
      \prod_{\nu=1}^\lambda \prod_{\deg \frakp = \nu \atop \frakp \text{ finite}} S(\frakp)(q^{-1}),
    \end{align*}
    whence $\tilde{E}_2(\lambda) = E_2(\lambda)$.
  \end{proof}
  
  Before discussing how to compute the $-\log S(\frakp)(q^{-1})$'s, we want to discuss the subject
  of numerical approximation. To compute $\log E_2(\lambda)$, we need to add a huge number of
  logarithms of rational numbers $\neq 1$, i.e. of transcendental numbers. But we have an advantage,
  namely all appearing logarithms are integral multiples of $\log(1 - q^{-i})$ for $i \in \{ 1,
  \dots, n \lambda \}$. In general, $n \lambda \ll q^\lambda$, whence it makes sense to write \[
  \log E_2(\lambda) = g \log q + \sum_{i=1}^{n \lambda} b_i \log(1 - q^{- i}) \] with $b_i \in \Z$,
  and to first compute the coefficients $b_i \in \Z$ -- for which no approximation is needed -- and
  then use the $b_i$ to compute an approximation of $\log E_2(\lambda)$. In particular, once we know
  $b_i$, it is easier to determine the precision of $\log(1 - q^{-i})$ that is required to compute
  $\log E_2(\lambda)$ with the wanted precision. Moreover, no floating point operation is required
  during the determination of the $b_i$, only integer arithmetic and finite field arithmetic. This
  improves the approach made in \cite{scheidlerstein-approxeulerprods}.
  
  Now, let us discuss how we can compute $S(\frakp)(t)$ for a place $\frakp \in \PP_{k(x)}$; for
  that, we use material from Sections~\ref{monogenicintbasessect} and \ref{factorization}. Let
  $d_\frakp := \gcd(n, \nu_\frakp(D))$. If $d_\frakp = 1$, we have $S(\frakp)(t) = (1 - t)
  \frac{1}{1 - t} = 1$ as $\frakp$ ramifies totally. In case $d_\frakp > 1$, let $F : \frako_\frakp
  \to \kappa(\frakp)$ be the projection, $\pi$ a uniformizer for $\frakp$ and $\alpha_\frakp := F(D
  \pi^{-\nu_\frakp(D)}) \in \kappa(\frakp)^*$. Then the factorization of $f_\frakp := Y^d -
  \alpha_\frakp \in \kappa(\frakp)[Y]$ determines $S(\frakp)$. We have seen that $f_\frakp$ is
  squarefree. In particular, we can effectively compute the $\deg \frakP$'s using
  Algorithm~\ref{splittingalgorithm}. We get the following algorithms and results:
  
  \begin{proposition}
    \label{logScompu}
    Assume that $k = \F_q$ is the exact constant field of $K$, and assume that $D \in k[x]$ and
    $\log n = \calO(\log q)$. Given a finite place~$\frakp \in \PP_{k(x)}$, the following algorithm
    computes the coefficients $a_i$ of $-\log S(\frakp)(q^{-1}) = \sum_{i=1}^n a_i \log(1 - q^{-i
    \deg \frakp})$ in \[ \calO((\deg \frakp)^2 \log^3 q \cdot (n + \deg \frakp \cdot (\log \log q +
    \log \deg \frakp)) + (\deg D)^2 \log^2 q) \] binary operations, assuming we know the
    factorization of $q^{\deg \frakp} - 1$ and the one of $p - 1$ for every prime~$p \mid (q^{\deg
    \frakp} - 1)$. For almost all places, the algorithm needs in fact just \[ \calO((\deg \frakp)^2
    \log^3 q \cdot (n + \deg \frakp \cdot (\log \log q + \log \deg \frakp)) + \deg \frakp \deg D
    \log^2 q) \] binary operations.
    \begin{algorithm}{Compute $-\log S(\frakp)(q^{-1})$ for a finite place $\frakp \in \PP_{k(x)}$}
      \label{logScompufinite}
      \alginput{$n$, $D \in \F_q[x]$, $\frakp$ given in form of an irreducible polynomial $p \in
      \F_q[x]$}
      \algoutput{$-\log S(\frakp)(q^{-1}) = \sum_{i=1}^n a_i \log(1 - q^{-i \deg \frakp})$ in terms
      of $a_1, \dots, a_n$}
      \begin{enum1}
        \item Set $t := 0$.
        \item Compute $D = q p + r$ with $q, r \in \F_q[x]$, $\deg r < \deg p$.
        \item If $r = 0$:
        \begin{enuma}
          \item Set $D := q$.
          \item Set $t := t + 1$.
          \item Go to Step~(2).
        \end{enuma}
        \item Compute $d := \gcd(n, t)$.
        \item If $d = 1$, return $a_i = 0$.
        \item Use Algorithm~\ref{splittingalgorithm} to compute the degrees of the irreducible
        factors of $Y^d - r$ in $(\F_q[x]/(p))[Y]$. Let $a_i$ be the number of irreducible factors
        of degree~$i$ over $\F_q[x]/(p)$.
        \item Set $a_1 := a_1 - 1$.
      \end{enum1}
    \end{algorithm}
    For the infinite place $\frakp_\infty$ of $k(x)$, the following algorithm computes the
    coefficients $a_i$ of $-\log S(\frakp)(q^{-1}) = \sum_{i=1}^n a_i \log(1 - q^{-i \deg \frakp})$
    in $\calO((n + \log \log q) \log^3 q)$~binary operations under the same assumptions as above:
    \begin{algorithm}{Compute $-\log S(\frakp)(q^{-1})$ for the infinite place of $k(x)$}
      \label{logScompuinfinite}
      \alginput{$n$, $D \in \F_q[x]$}
      \algoutput{$-\log S(\frakp)(q^{-1}) = \sum_{i=1}^n a_i \log(1 - q^{-i \deg \frakp})$ in terms
      of $a_1, \dots, a_n$}
      \begin{enum1}
        \item Compute $d := \gcd(n, \deg D)$.
        \item If $d = 1$, return $a_i = 0$.
        \item Use Algorithm~\ref{splittingalgorithm} to compute the degrees of the irreducible
        factors of $Y^d - \sgn(D)$ in $\F_q[Y]$. Let $a_i$ be the number of irreducible factors
        of degree~$i$ over $\F_q$.
        \item Set $a_1 := a_1 - 1$.
      \end{enum1}
    \end{algorithm}
  \end{proposition}
  
  \begin{proof}
    Let us first estimate the time required to compute~$\alpha = D p^{-\nu_\frakp(D)} \mod p$ in
    Algorithm~\ref{logScompufinite}, i.e. for Steps~(1)--(3). We have $\nu_\frakp(D) = \calO(\deg D
    / \deg p)$, and in the $t$-th (beginning with $t = 0$) iteration of Step~(2) we have $\deg D_t =
    \deg D_0 - t \deg p$, where $D_t$ is the value of $D$ in the $t$-th iteration. Computing a long
    division needs $\calO(\deg p \deg D_t)$ operations in $\F_q$ \cite[Section~3.1.3]{cohen}, whence
    we need $\calO((\deg D / \deg p)^2 (\deg p)^2) = \calO((\deg D)^2)$ operations in $\F_q$ to
    obtain $t = \nu_\frakp(D)$ and $r = D p^{-t} \mymod p$. Note that in case $t = 0$, we just need
    $\calO(\deg D \cdot \deg p)$ operations. Moreover, note that for
    Algorithm~\ref{logScompuinfinite}, the corresponding steps need no time as $\alpha = \sgn(D)$
    and $\nu_\frakp(D) = -\deg D$.
    
    The finite field $\F_{\tilde{q}}$ in Algorithm~\ref{splittingalgorithm} is $\kappa(\frakp) \cong
    \F_q[x] / \ggen{p}$, whence it has $q^{\deg p} = q^{\deg \frakp}$~elements. Hence, by
    Proposition~\ref{splittingalgorithmruntime}, the computation of the $\deg \frakP$'s requires
    $\calO(n \log^3 n + n \log n \cdot (\deg \frakp)^2 \log^2 q + (\deg \frakp)^3 \log^3 q \cdot
    (\log \log q + \log \deg \frakp))$ binary operations.  Using $\log n = \calO(\log q)$, this
    simplifies to $\calO((\deg \frakp)^2 \log^3 q \cdot (n + \deg \frakp \cdot (\log \log q + \log
    \deg \frakp)))$.
  \end{proof}
  
  Next, we want to discuss the question on how to enumerate all monic irreducible polynomials~$p \in
  \F_q[x]$ with $\deg p \le \lambda$. It is well-known that there are $\frac{q^d}{d} +
  \calO(q^{d/2}/d)$ monic irreducible polynomials in $\F_q[x]$ of degree~$d$. (For $d = 1$, every
  polynomial is irreducible.) The total number of monic polynomials of degree~$d$ is $q^d$, whence
  it makes sense to try all monic polynomials and use a irreducibility test; according to
  \cite[Section~14.9]{vzgathen-moderncomputeralgebra}, this can be done in $\calO((M(d) \log q +
  (d^{1.688} + d^{1/2} M(d)) \delta(d) \log d) \log^2 q)$~binary operations; using Karatsuba's
  method, one has $M(d) = \calO(d^{1.59})$ \cite[p.~242]{vzgathen-moderncomputeralgebra}, and
  $\delta(d) < \log_2 d$, whence we get a total running time of \[ \calO(d^{1.59} \log^3 q +
  d^{2.09} \log^2 d \cdot \log^2 q) \] for one irreducibility test. Hence, we get a running time
  of \[ \calO(q^d d^{1.59} \log^3 q + q^d d^{2.09} \log^2 d \cdot \log^2 q) \] to enumerate all
  monic irreducible polynomials of degree~$d$ over $\F_q$. Note that the running time for one
  irreducibility check is dominated by the running time for Algorithm~\ref{logScompufinite} applied
  to any irreducible polynomial of degree~$d$.
  
  \begin{theorem}
    Assume that $k = \F_q$ is the exact constant field of $K$. Let $\lambda \in \{ 1, 2, \dots, g
    \}$. The following algorithm computes the Euler product approximation \[ E_2(\lambda) =
    \prod_{\frakp \in \PP_{k(x)} \atop \deg \frakp \le \lambda} S(\frakp)(1/q) \] in \[
    \calO(\lambda q^\lambda [\lambda n \log q + \lambda^2 \log q \cdot \log (\lambda \log q) +
    D_{\deg} + \lambda^{2.09} \log^2 \lambda] \log^2 q) \] binary operations, assuming $\log n =
    \calO(\log q)$:
    \begin{algorithm}{Compute the Euler product approximation $E_2(\lambda)$}
      \alginput{$n$, $\lambda$, and the squarefree decomposition $D = \sgn(D)
      \prod_{i=-\infty}^\infty f_i^i$ of $D$} \algoutput{$E_2(\lambda)$ for the function field $K :
      y^n = D(x)$}
      \begin{enum1}
        \item Compute $\tilde{D} := \sgn(D) \cdot \prod_{i=-\infty}^\infty f_i^{i \mymod n} \in
        \F_q[x]$. When calling the algorithms to compute $-\log S(\frakp)(q^{-1})$, use $\tilde{D}$
        instead of $D$.
        \item Set $a_1 := \dots := a_{\lambda n} := 0$.
        \item For $\nu = 1, \dots, \lambda$ do:
        \begin{enuma}
          \item Factor $\nu$, factor $q^\nu - 1$ and, for every prime~$p$ dividing $q^\nu - 1$,
          factor $p - 1$.
          \item If $\nu = 1$, compute $-\log S(\frakp_\infty)(q^{-1}) = \sum_{i=1}^n b_i \log(1 -
          q^{-i})$, where $\frakp_\infty$ is the infinite place of $k(x)$, using
          Algorithm~\ref{logScompuinfinite}, and set $a_i := b_i$, $1 \le i \le n$.
          \item For every monic polynomial $f \in \F_q[x]$ with $\deg f = \nu$ do:
          \begin{enum1}
            \item Test whether $f$ is irreducible; if this is not the case, continue with the next
            choice of $f$.
            \item Compute $-\log S(\frakp)(q^{-1}) = \sum_{i=1}^n b_i \log(1 - q^{-i \nu})$, where
            $\frakp$ is the finite place of $k(x)$ belonging to $f$, using
            Algorithm~\ref{logScompufinite}, and set $a_{\nu i} := a_{\nu i} + b_i$, $1 \le i \le
            n$.
          \end{enum1}
        \end{enuma}
        \item Compute $r := g \log q - \sum_{i=1}^{n \lambda} a_i \log(1 - q^{-i})$.
        \item Compute and return $\exp(r)$.
      \end{enum1}
    \end{algorithm}
  \end{theorem}
  
  \begin{proof}
    The correctness follows from the previous discussion. What is left is to estimate the running
    time. First, note that for almost all computations of $-\log S(\frakp)(q^{-1})$, we have
    $\nu_\frakp(D) = 0$. The number of possible exceptions is bounded by $1 +
    \sum_{i=-\infty}^\infty \deg f_i = \calO(D_{\deg})$; these are exactly the places of $k(x)$
    which ramify in $K$.
    
    We can ignore the running time required for the infinite place of $k(x)$, as there is only one,
    compared to the $q$ finite places of degree~one. Moreover, the factorization in Step~(3\;a) can
    be ignored.
    
    We have seen above that there are $\frac{1}{\nu} q^{\nu} + \calO(q^{\nu/2})$ monic irreducible
    polynomials of degree~$\nu$, whence the time spent in Step~(3\;c\;2) for a fixed~$\nu$ equals
    $\calO(q^\nu [\nu n \log q + \nu^2 \log q \cdot \log (\nu \log q) + D_{\deg}] \log^2 q)$~binary
    operations (in all but at most $\calO(D_{\deg})$ cases).
    
    In contrast, the time to enumerate all these polynomials is $\calO(q^\nu \nu^{1.59} (\log q)^3 +
    q^\nu \nu^{2.09} (\log \nu)^2 (\log q)^2)$. Therefore, the total running time for Step~(3\;c) is 
    \[ \calO(q^\nu [\nu n \log q + \nu^2 \log q \cdot \log (\nu \log q) + D_{\deg} + \nu^{2.09}
    \log^2 \nu] \log^2 q). \] Hence, one obtains a total running time of \[ \calO(\lambda q^\lambda
    [\lambda n \log q + \lambda^2 \log q \cdot \log (\lambda \log q) + D_{\deg} + \lambda^{2.09}
    \log^2 \lambda] \log^2 q) \] for Step~(3).
    
    Finally, we have to estimate the running time for Steps~(4) and (5). All $a_i$ are bounded by
    $n$~times the number of places of $k(x)$ of degree~$\le \lambda$; hence, $\abs{a_i} \le
    n \sum_{j=0}^\lambda q^j \le n \lambda q^\lambda$.
    
    Moreover, by Theorem~\ref{eulerapproxthm} and the Hasse-Weil bounds $E_2(\lambda)$ is of order
    of magnitude $q^g$. To compute $E_2(\lambda)$ with error $< \frac{1}{2}$, we therefore need to
    compute $\log E_2(\lambda)$ with error $< \log(1 + \frac{1}{2 (1 + \sqrt{q})^{2 g}})$. We have
    $n \lambda + 1$ terms to add for $\log E_2(\lambda)$, whence it suffices to compute each term
    with error $< \frac{\log(1 + \frac{1}{2 (1 + \sqrt{q})^{2 g}})}{1 + n \lambda}$. Each term can
    be bounded by $g \log q$ resp. $n \lambda q^\lambda$, whence we need a precision of at most
    $\log_2 (n \lambda q^\lambda) - \log_2 \frac{\log(1 + \frac{1}{2 (1 + \sqrt{q})^{2 g}})}{1 + n
    \lambda} \le \lambda \log_2 q + 2 \log_2(n \lambda) + 2 + 2 g \log_2 (1 + \sqrt{q}) = \calO(g
    \log q)$ bits as $\lambda \le g$. In particular, the computational costs for computing the
    approximation of $E_2(\lambda)$ out of the $a_i$'s are polynomial in $g \log q$ and are
    irrelevant compared to the costs of Step~(3).
  \end{proof}
  
  If we assume that $n$, $D_{\deg}$ and $\lambda$ stay bounded while $q$ grows, we obtain a running
  time of $\calO(q^\lambda \log^3 q \cdot \log \log q)$ binary operations. This makes the statements
  on the running time of computing Euler product approximations in
  \cite{scheidlerstein-approxeulerprods} more precise for the case of radical function fields.
  
  \section{Conclusion}
  
  In this paper, we have described explicit methods which allow to implement arithmetic in radical
  function fields. We have presented methods to
  \begin{enuma}
    \item compute integral bases for $\calO_\frakp'$, $\frakp \in \PP_{k(x)}$;
    \item compute all places of $K$ lying above $\frakp \in \PP_{k(x)}$ as well as generators of
    their corresponding prime ideal in $\calO_\frakp'$;
    \item compute a simple integral basis for $\calO$, the integral closure of $k[x]$ in $K$;
    \item compute Riemann-Roch spaces;
    \item compute the exact constant field, its degree and a generator of it over $k$;
    \item compute the genus;
    \item approximate the divisor class number~$\abs{\Pic^0(K)}$ using an Euler product
    approximation.
  \end{enuma}
  The integral bases are given in a very explicit form; they can be written down knowing only $n$, a
  uniformizer~$\pi$, $\nu_\frakp(D)$, respectively $n$ and the squarefree factorization of $D$. For
  most computations, bounds on the running time and storage space are given.
  
  This allows to implement arithmetic in radical funtion field, assuming that a library for working
  with polynomials over finite fields such as NTL\footnote{A C++ library by V.~Shoup for doing
  Number Theory. See also \url{http://www.shoup.net/ntl/}.} is available. Using infrastructure
  methods (see \cite{ff-tioagfoaur}) or the methods described in \cite[Section~2]{diem-habil}, one
  can do effective arithmetic in the divisor class group $\Pic^0(K)$.
  
  Some of these results were already known in special cases; for example, in case $D$ is a
  squarefree polynomial and $\gcd(\deg D, n) = 1$, the function field is superelliptic, and
  arithmetic in it is described in \cite{galbraith-paulus-smart}. Under the assumption that $D$ is a
  polynomial not divisible by any $n$-th power, the formula for an integral basis was given in
  \cite{wu-radicalffsintbases}. Our approach generalizes both results. The Riemann-Roch space
  computation was described for general function fields \cite{hessRR}, as well as the Euler product
  approximation \cite{scheidlerstein-approxeulerprods}. Our approach makes the running time bounds
  more precise, and in the case of the Euler product approximation, improves on the original
  algorithm by making it more robust to approximation errors, as well as easier to implement as one
  does not have to compute the $z_j(\frakp)$'s as well as handle the $S_\nu(i)$'s and the infinite
  series involving them.

\newcommand{\etalchar}[1]{$^{#1}$}
\providecommand{\bysame}{\leavevmode\hbox to3em{\hrulefill}\thinspace}
\providecommand{\MR}{\relax\ifhmode\unskip\space\fi MR }
\providecommand{\MRhref}[2]{%
  \href{http://www.ams.org/mathscinet-getitem?mr=#1}{#2}
}
\providecommand{\href}[2]{#2}


\begin{thebibliography}{CFA{\etalchar{+}}06}

\bibitem[Bau04]{mark-arithmeticcubicfields}
Mark~L. Bauer, \emph{The arithmetic of certain cubic function fields}, Math.
  Comp. \textbf{73} (2004), no.~245, 387--413 (electronic). \MR{MR2034129
  (2004k:11179)}

\bibitem[CFA{\etalchar{+}}06]{hehcc}
H.~Cohen, G.~Frey, R.~Avanzi, C.~Doche, T.~Lange, K.~Nguyen, and F.~Vercauteren
  (eds.), \emph{Handbook of elliptic and hyperelliptic curve cryptography},
  Discrete Mathematics and its Applications (Boca Raton), Chapman \& Hall/CRC,
  Boca Raton, FL, 2006. \MR{MR2162716 (2007f:14020)}

\bibitem[Coh96]{cohen}
H.~Cohen, \emph{A course in computational algebraic number theory}, third
  corrected ed., Graduate Texts in Mathematics, vol. 138, Springer-Verlag,
  Berlin, 1996. \MR{MR1228206 (94i:11105)}

\bibitem[Die08]{diem-habil}
Claus Diem, \emph{On arithmetic and the discrete logarithm problem in class
  groups of curves}, Habilitationsschrift, May 2008.

\bibitem[Fon09]{ff-tioagfoaur}
F.~Fontein, \emph{The infrastructure of a global field of arbitrary unit rank},
  2009, Submitted to Math. Comp. Preprint available at
  \verb"http://arxiv.org/abs/0809.1685".

\bibitem[GPS02]{galbraith-paulus-smart}
S.~D. Galbraith, S.~M. Paulus, and N.~P. Smart, \emph{Arithmetic on
  superelliptic curves}, Math. Comp. \textbf{71} (2002), no.~237, 393--405
  (electronic). \MR{MR1863009 (2002h:14102)}

\bibitem[He{\ss}02]{hessRR}
F.~He{\ss}, \emph{Computing {R}iemann-{R}och spaces in algebraic function
  fields and related topics}, J. Symbolic Comput. \textbf{33} (2002), no.~4,
  425--445. \MR{MR1890579 (2003j:14032)}

\bibitem[Pau98]{paulusReduct}
S.~M. Paulus, \emph{Lattice basis reduction in function fields}, Algorithmic
  number theory (Portland, OR, 1998) (Berlin), Lecture Notes in Comput. Sci.,
  vol. 1423, Springer, 1998, pp.~567--575. \MR{MR1726102 (2000i:11193)}

\bibitem[Sch01]{scheidler-infrastructurepurelycubic}
R.~Scheidler, \emph{Ideal arithmetic and infrastructure in purely cubic
  function fields}, J. Th\'eor. Nombres Bordeaux \textbf{13} (2001), no.~2,
  609--631. \MR{MR1879675 (2002k:11209)}

\bibitem[SS09]{scheidlerstein-approxeulerprods}
R.~Scheidler and A.~Stein, \emph{Approximating euler products and class number
  computation in algebraic function fields}, To appear in Rocky Mountain
  Journal of Mathematics (2009).

\bibitem[Sti93]{stichtenoth}
H.~Stichtenoth, \emph{Algebraic function fields and codes}, Universitext,
  Springer-Verlag, Berlin, 1993. \MR{MR1251961 (94k:14016)}

\bibitem[Sut07]{sutherland-phd}
Andrew~V. Sutherland, \emph{Order computations in generic groups}, Ph.{D}.
  thesis, Massachusetts Institute of Technology, 2007.

\bibitem[vzGG03]{vzgathen-moderncomputeralgebra}
Joachim von~zur Gathen and J{\"u}rgen Gerhard, \emph{Modern computer algebra},
  second ed., Cambridge University Press, Cambridge, 2003. \MR{MR2001757
  (2004g:68202)}

\bibitem[Wu09]{wu-radicalffsintbases}
Qingquan Wu, \emph{Explicit construction of integral bases of radical function
  fields}, To appear in J. Th\'eor. Nombres Bordeaux (2009).

\end{thebibliography}
\end{document}